\documentclass[stsy]{informs-stsy}

\usepackage[normalem]{ulem}
\usepackage{xcolor}

\usepackage{cancel}

\OneAndAHalfSpacedXI

\usepackage{amsmath,amsgen,amstext,amssymb}
\usepackage{amsfonts}
\usepackage{dsfont}
\usepackage{mathrsfs}
\usepackage{latexsym}
\usepackage{times}
\usepackage{graphics}
\usepackage{graphicx}
\usepackage{epsfig}
\usepackage{psfrag}
\usepackage{bm}
\usepackage{subcaption}
\usepackage{algorithm}

\DeclareMathAlphabet{\mathpzc}{OT1}{pzc}{m}{it}
\newcommand{\bq}{\mathbf{q}}
\newcommand{\brho}{\bm{\rho}}
\newcommand{\bmu}{\bm{\mu}}
\newcommand{\bnu}{\bm{\nu}}
\newcommand{\bomega}{\bm{\omega}}

\newcommand{\E}{\mathbb{E}}
\newcommand{\ex}{\mathbb{E}}

\newcommand{\norm}[1]{\| #1 \|}
\newcommand{\indic}{\mathbb{I}}

\newcommand{\beqn}{\begin{equation}}
\newcommand{\eeqn}{\end{equation}}

\newcommand\Reals{{\mathbb{R}}}
\newcommand\Ints{{\mathbb{Z}}}
\newcommand{\Nats}{{\mathbb{N}}}

\newcommand{\eqdef}{{:=}}

\providecommand{\norm}[1]{{\lVert#1\rVert}}

\newcommand{\bA}{{\bm{A}}}

\newcommand{\bQ}{{\mathbf{Q}}}

\newcommand{\ba}{{\mathbf{a}}}

\newcommand{\bfe}{{\mathbf{e}}}

\newcommand{\bw}{{\mathbf{w}}}
\newcommand{\bx}{{\mathbf{x}}}
\newcommand{\bs}{{\mathbf{s}}}
\newcommand{\bS}{{\bm{S}}}
\newcommand{\by}{{\mathbf{y}}}
\newcommand{\bz}{{\mathbf{z}}}

\newcommand{\cA}{{\mathcal{A}}}

\newcommand{\cC}{{\mathcal{C}}}

\newcommand{\cM}{{\mathcal{M}}}

\newcommand{\cP}{{\mathcal{P}}}

\newcommand{\cV}{{\mathcal{V}}}

\newcommand{\IndSet}{\mathcal{I}}
\newcommand{\Int}{\mathbb{Z}}
\newcommand{\bzero}{\mathbf{0}}

\def\ba#1\ea{\begin{align*}#1\end{align*}}
\def\ban#1\ean{\begin{align}#1\end{align}}
\newcommand{\todo}[1]{}

\newcommand{\vlam}{\ensuremath{{\mbox{\boldmath{$\lambda$}}}} }

\newcommand{\siva}[1]{\textcolor{brown}{\textsc{!!SIVA:} #1 !! }}
\newcommand{\mss}[1]{{\textcolor{black}{#1}}}

\usepackage[numbers]{natbib}
 \bibpunct[, ]{(}{)}{,}{a}{}{,}%
 \def\bibfont{\small}%
\TheoremsNumberedThrough     
\ECRepeatTheorems

\EquationsNumberedThrough    

\MANUSCRIPTNO{}

\begin{document}


\RUNAUTHOR{Lu et al.}

\RUNTITLE{Delay-Optimal Scheduling in Input-Queued Switches}

\TITLE{Optimal Weighted-Delay Scheduling in \\{\LARGE{$2\times 2$}} Input-Queued Switches}

\ARTICLEAUTHORS{%
\AUTHOR{Yingdong Lu}
\AFF{Mathematical Sciences Department, IBM Research, Yorktown Heights, NY 10598, USA, \EMAIL{yingdong@us.ibm.com}} 
\AUTHOR{Siva Theja Maguluri}
\AFF{School of Industrial and Systems Engineering, Georgia Institute of Technology, Atlanta, GA 30332, USA, \EMAIL{siva.theja@gatech.edu}}
\AUTHOR{Mark S.\ Squillante}
\AFF{Mathematical Sciences Department, IBM Research, Yorktown Heights, NY 10598, USA, \EMAIL{mss@us.ibm.com}} 
\AUTHOR{Tonghoon Suk}
\AFF{Digital Technology Department, Otis Elevator Company,  Farmington, CT 06032, USA, \EMAIL{tonghoon.suk@otis.com}} 
\AUTHOR{Xiaofan Wu}
\AFF{School of Industrial and Systems Engineering, Georgia Institute of Technology, Atlanta, GA 30332, USA, \EMAIL{xwu319@gatech.edu}}
} 

\ABSTRACT{%
Motivated by few delay-optimal scheduling results, in comparison to results on throughput optimality, we investigate a canonical input-queued switch scheduling problem in which the objective is to minimize the discounted delay cost over an infinite time horizon. We derive an optimal scheduling policy and establish corresponding theoretical properties, which are expected to be of interest more broadly than input-queued switches. Computational experiments demonstrate and quantify the benefits of our optimal scheduling policy over alternative policies such as variants of MaxWeight scheduling, well-known to be throughput optimal and more recently shown to be delay optimal in the heavy-traffic regime limit.
}%


\KEYWORDS{Optimal Control, Markov Decision Processes, Input-Queued Switches.}
\HISTORY{.}

\maketitle

%


\section{Introduction}
Input-queued switch architectures are widely used in modern computer and communication networks. The analysis and control of these high-speed, low-latency switch networks is critical for understanding fundamental design and performance issues related to internet routers, cloud computing data centers, and high-performance computing. A large and rich literature exists around scheduling in these computer and communication systems. Moreover, input-queued switches represent an important mathematical model for studying theoretical scheduling issues of broad interest.

Most of the previous research related to scheduling in input-queued switches has focused on optimal throughput. In particular, the MaxWeight scheduling policy, first introduced by \cite{TasEph92} for wireless networks and subsequently by \cite{MCKEOWN96} specifically for input-queued switches,
is well-known for being throughput optimal. The issue of delay-optimal scheduling for switches, however, is far less clear with much more limited results. This is not surprising given that the study of delays (or equivalently, via Little's Law, queue lengths) in these types of stochastic systems is hard in general. Hence, previous research on optimal delay scheduling in input-queued switches has focused on heavy-traffic and related asymptotic regimes; see, e.g., \cite{Stolyar_cone_SSC, ShaWis_11, kang2009diffusion,shah2012optimal,zhong2014scaling}.

Such previous work includes \cite{stolyar2004} establishing that the MaxWeight scheduling policy is asymptotically optimal in heavy traffic for an objective function of the summation of the squares of the queue lengths with the assumption of complete resource pooling; and \cite{MagSri_SSY16_Switch} showing that MaxWeight scheduling is optimal in heavy traffic for an objective function of the summation of the queue lengths, under the
assumption that all the ports are saturated; with these latter results subsequently extended by \cite{maguluri2016b} to the case of incompletely saturated ports, though still for the summation of the average queue lengths.
Nevertheless, beyond these and related recent results limited to the heavy-traffic regime, the question of delay-optimal scheduling in input-queued switches remains open in general, as does the question of optimal scheduling under even more general objective functions (such as those considered herein).

In this paper, we seek to gain fundamental insights on optimal delay-cost scheduling in these stochastic systems by focusing on the canonical $2 \times 2$ input-queued switch model.
The objective of the corresponding stochastic control problem is to determine the scheduling policy that minimizes the discounted summation over general linear cost functions of the
expected delays (queue lengths) associated with each queue.
Using well-known results (e.g.,~\cite{O-L-book,sennott-book}),
the optimal scheduling policy for the discounted-cost setting can be used to obtain an optimal policy for the corresponding average-cost setting.
Our derivation of an optimal solution consists of first partitioning the state space into three different scheduling decision regions of interest, namely the interior, the trivial boundary, and the critical boundary.
We then derive an optimal policy and establish structural properties of the associated value function for each of these regions, in particular showing that the optimal solution in the interior follows the $c\mu$ policy. 
Meanwhile, for all other regimes except for trivial regions of the boundary of the state space wherein the optimal decisions are obvious,
we establish that an optimal policy either follows the $c\mu$ policy or consists of a switching curve that takes into account the arrival processes.
We derive explicit expressions for the optimal switching curve under a scenario consisting of symmetric Bernoulli arrivals and unit costs.
More generally, we exploit our theoretical results on the switching curve to devise an approach to approximate the optimal policy and we show that this renders an asymptotically optimal policy.
Our optimal cost-weighted-delay scheduling analysis and results for the $2\times 2$ switch are, to the best of our knowledge, the first set of theoretical results on an optimal
scheduling policy to minimize the expected weighted queue lengths in general for a non-trivial switch;
and the fundamental insights gained therefrom motivate our ongoing analysis and results for the general system in the sequel.
These insights include fundamental differences in the decisions made under our optimal policy and those under the MaxWeight policy which shows that the latter is not delay optimal in general.

Given these important differences in decisions between our optimal scheduling policy and the MaxWeight scheduling policy, together with recent results on the queue-length
(delay) optimality of the latter in the heavy-traffic regime limit, we conduct numerous computational experiments to shed further light on various important theoretical
issues with respect to optimal delay-cost scheduling in input-queued switches.
In the case of symmetric arrivals and unit costs, our optimal solution renders an explicit optimal scheduling policy and the corresponding computational experiments demonstrate significant differences in steady-state queue-length performance between our optimal policy and MaxWeight scheduling, further
supporting that the MaxWeight policy
is not
delay optimal in general.
From these results we also observe a gap between the performance of our optimal solution and the corresponding weight-based universal lower bound established by~\cite{LuMaSq+18},
suggesting that the lower bound may not be tight. 
More generally, our optimal solution involves a switching curve in some cases of the critical boundary region, for which we exploit our theoretical results to obtain a look-ahead
policy that is proven to be asymptotically optimal.
The corresponding computational experiments indicate that the look-ahead policy converges quickly and outperforms both MaxWeight scheduling and one of its related variants, where the optimality gap varies from $7\%$ to $16\%$ depending on the experimental settings.

A preliminary analysis of a very special case of the $2\times 2$ input-queued switch, restricted solely to symmetric arrivals and unit costs, has appeared in a conference proceedings (without copyright transfer); refer to \citep{LuMaSq+17}.
The current paper significantly extends this preliminary conference paper in several important aspects, including our complete and thorough derivation of a solution to the delay-cost optimal scheduling problem under general (linear) costs and general arrivals for the canonical $2 \times 2$ input-queued switch and an expanded set of computational experiments that quantitatively evaluate our optimal scheduling policy and explores issues with respect to the optimality of MaxWeight scheduling in general.

The remainder of this paper is organized as follows.
We first present in Section~\ref{sec:model} some technical preliminaries, including our mathematical model for the canonical $2 \times 2$ input-queued switch, our formulation of the corresponding mathematical optimization problem, and an important equivalent problem formulation.
We then present in Section~\ref{sec:optimal} our analysis and results for optimal delay-cost scheduling and related structural properties,
with many of our proofs collected in Section~\ref{sec:proofs}.
Computational experiments results are presented in Section~\ref{sec:experiments}, followed by concluding remarks.
Additional proofs are provided in 
the Appendices.

\section{Technical Preliminaries}\label{sec:model}

\subsection{Mathematical Model}
Consider an input-queued switch with $2$ input ports, $2$ output ports, and a queue associated with each input-output port combination
that stores packets (customers) waiting to be transmitted from the input port to the output port.
Let $(i,j) \in \IndSet \eqdef \{ (i,j) : i,j\in\{1,2\} \}$ index the queue associated with input port $i$ and output port $j$.
Packets arrive at queue $(i,j)$ according to an exogenous stochastic process. 
All packets are assumed to be of the same size and require exactly one unit of service.

Time is slotted and denoted by a nonnegative integer $t\in\Int_+ \eqdef \{0,1,\ldots\}$.
At each time $t$, a \emph{schedule} refers to a subset of active queues that satisfies the constraints:
(1) At most one queue of each input port can be active;
(2) At most one queue of each output port can be active.
A maximal schedule refers to a schedule with exactly one queue for each input and exactly one queue for each output being active.
A scheduling policy selects a schedule from which to simultaneously transmit packets.
Formally, a schedule can be described by a $4$-dimensional binary vector $\bs=(s_{\brho})_{\brho\in\IndSet}$ such that
$s_{\brho}=1$ if queue $\brho$ is in the schedule, and $s_{\brho}=0$ otherwise;
if one of the activated queues in a schedule is empty, then the provided service at this queue is simply assumed to be unused.
Denote by $\cP$ the set of all possible schedules.

We study the scheduling problem in which a policy $\pi$ selects a schedule $\bS^\pi(t) \eqdef (S^\pi_{\brho}(t))_{\brho\in\IndSet} \in \cP$
in every time slot $t$.
Let $Q^\pi_{\brho}(t) \in \Int_+$ denote the length of queue $\brho$ at time $t$ under policy $\pi$ and $A_{\brho}(t) \in \Int_+$ the number of arrivals to queue $\brho$ during $[t,t+1)$.
We assume that $\{ A_{\brho}(t) : t \in \Int_+, \, \brho \in \IndSet \}$ are independent random variables and that, for fixed $\brho \in \IndSet$,
$\{ A_{\brho}(t) : t \in \Int_+ \}$ are identically distributed with $\lambda_{\brho} = \E[A_{\brho}(t)]$ the arrival rate of queue $\brho$.
Define $\bQ^\pi(t) \eqdef (Q^\pi_{\brho}(t))_{\brho\in\IndSet}$ and $\bA(t) \eqdef (A_{\brho}(t))_{\brho\in\IndSet}$.
Since each arrival process is independent and identically distributed (i.i.d.), we use $\bA$ to denote a random variable that has the same joint distribution as $\bA(t)$ for any $t$.
Within the time slot, service happens before arrivals.
The queueing dynamics under policy $\pi$ can then be expressed as
\begin{align}
Q^\pi_{\brho}(t+1) & =  Q^\pi_{\brho}(t) + A_{\brho}(t) - S^\pi_{\brho}(t) \cdot \indic_{\{Q^\pi_{\brho}(t)>0\}}, \label{eq:dynamics_of_Q} \\
& =
[Q^\pi_{\brho}(t) -  S^\pi_{\brho}(t) ]^+ + A_{\brho}(t)
\nonumber
\end{align}
where $\indic_\cA$ denotes an indicator function associated with event $\cA$, returning $1$ if $\cA$ is true and $0$ otherwise,
and $[x]^+$ denotes $\max\{x,0\}$.

It is well known, see, e.g.,~\cite{stolyar2004maxweight}, that the capacity region of the switch is given by
\begin{align*}
\cC = \left\{\vlam: \sum_i\lambda_{(i,j)}<1 \text{ and }  \sum_j\lambda_{(i,j)}<1\right\},
\end{align*}
i.e., the total arrival rate to each input port and each output port can be at most one.
Hence, as long as the arrival rates are in the capacity region $\cC$, there exists a scheduling policy under which the underlying Markov chain
of the queue length processes $\bQ^\pi(t)$ is positive recurrent.
Moreover, for any arrival rate outside the set $\cC$, no scheduling policy can lead to a positive recurrent $\bQ^\pi(t)$.
We therefore assume throughout that the arrival rates are within the capacity region $\cC$.

\subsection{Formulation of Mathematical Optimization}
Our goal is to establish an optimal scheduling policy that minimizes the total discounted delay cost over an infinite time horizon. 
Given the relationship between delays and queue lengths via Little's Law, we henceforth focus on cost as a function of the queue lengths.
More specifically, the cost under policy $\pi$ at time $t$ is a linear function of the total queue length at time $t$, namely
\begin{equation*}
  c^\pi(t) \; = \; \sum_{\brho\in\IndSet}c_{\brho} Q^\pi_{\brho}(t)
\end{equation*}
for the per-queue cost function constants $c_{\brho}$.
We are interested in the total discounted cost over an infinite horizon given by
\begin{equation*}
  J_{\beta}(\bq,\pi) \eqdef \sum_{t=0}^\infty \E[\beta^t\,c^\pi(t)],
\end{equation*}
with initial queue length vector $\bQ^\pi(0)= \bq$, discount factor $\beta\in(0,1)$, and $\bQ^\pi(t)$ following \eqref{eq:dynamics_of_Q}.

Observe from \eqref{eq:dynamics_of_Q} that $\bQ^\pi(t+1)$ is determined by $\bS^\pi(t)$, which is under the control of the scheduling policy.
A scheduling policy is called \emph{admissible} if the schedule $\bS^\pi(t)$ at time $t$ is based solely on information revealed up to time $t$,
such as $\bS^\pi(t')$, $\bQ^\pi(t'+1)$, and $\bA(t')$ for all $t'<t$.
It follows from known results in Markov decision process theory (see, e.g.,~\cite{Pute05,Bert12}) that there exists an optimal stationary Markov policy
in which $\bS^\pi(t)$ depends only on $\bQ^\pi(t)$ (and not even on time $t$),
and therefore we restrict our attention herein to such stationary Markov scheduling policies.
With a slight abuse of notation, we use $\bS^\pi(\bq)$ to denote the schedule under policy $\pi$ in state $\bq$.
Let $\cM$ denote the set of all stationary Markov policies.
Then, more formally, we seek to solve the scheduling optimization problem
\begin{equation}
  \label{prob:opt_J}
  \tag{$\textrm{P}_{\beta}$}
  \min_{\pi \in \cM} J_{\beta}(\bq,\pi)
\end{equation}
and find an optimal policy $\pi^*$ that achieves the minimum.

\subsection{Maximizing Service Rate}
\label{sec:alt_formulation}
The cost for each time period in problem~\eqref{prob:opt_J} depends on the current queue lengths which involve both the arrival and service processes.
Thus, instead of directly solving~\eqref{prob:opt_J},
we shall consider an equivalent problem that is based on a reward for maximizing the service rate, where
the reward only depends on the current queue lengths and the service action.
In particular, upon choosing schedule $\bs\in \cP$ with current queue length vector $\bq\in\Int_+^{|\IndSet|}$ and $|\IndSet|=4$,
the reward function $r:\Int_+^{4}\times \cP\to\Reals_+$ is defined by
\begin{equation*}
  r(\bq,\bs) \; \eqdef \; \sum_{\brho\in\IndSet}c_{\brho} s_{\brho} \cdot \indic_{\{q_{\brho}>0\}} .
\end{equation*}
The  corresponding discounted infinite horizon reward function under the stationary policy $\pi$ is defined as
\begin{align*}
  \tilde{J}_{\beta}(\bq, \pi)~&\eqdef~\E\left[ \sum_{t=0}^\infty \beta^t\,r(\bQ^\pi(t),\bS^\pi(t)) \right]
\end{align*}
where $\bQ^\pi(0)=\bq$ is the initial state.
Then we can construct an alternative optimization problem as follows:
\begin{equation}
  \label{prob:opt_tilde_J}
  \tag{$\tilde{\textrm{P}}_{\beta}$}
  \max_{\pi \in \cM} \tilde{J}_{\beta}(\bq, \pi).
\end{equation}

Next, we show that if there is an optimal (stationary) policy $\pi^*$ of \eqref{prob:opt_tilde_J}, then $\pi^*$ is an optimal policy of \eqref{prob:opt_J},
noting that a similar result was proved by
\cite{baras_two} and \cite{baras_general} for a very different parallel queueing system.
%
\begin{proposition}\label{prop:alt_formulation}
For any $\beta \in (0,1)$, any policy $\pi\in \cM$ that is an optimal solution for problem~\eqref{prob:opt_tilde_J} is also an optimal solution for problem~\eqref{prob:opt_J}, and vice versa.
\end{proposition}
\proof{Proof.}
From the queueing dynamics in~\eqref{eq:dynamics_of_Q} and the definition of the cost and reward functions, we have
\ba
c^\pi(t+1) = c^\pi(t)+\sum_{\brho\in\IndSet}c_{\brho} A_{\brho}(t)- r^\pi(t),
\ea
where $r^\pi(t) \eqdef r(\bQ^\pi(t),\bS^\pi(t))$.
Summing over $t$ and discounting with $\beta$ then yields
\begin{align*}
  J_\beta(\bq,\pi) &= c^\pi(0)+\beta \sum_{t=0}^\infty \E\left[ \beta^t c^\pi(t+1) \right]
= c^\pi(0)+\beta \sum_{t=0}^\infty \E\left[ \beta^t \left(c^\pi(t)+\sum_{\brho\in\IndSet}c_{\brho} A_{\brho}(t)- r^\pi(t)\right) \right] \\
  &= c^\pi(0)+\beta J_\beta(\bq,\pi)+g-\beta \tilde{J}_\beta(\bq,\pi),
\end{align*}
with
\begin{align}\label{eq:g-constant}
g=\sum_{t=0}^\infty\beta^{t+1}\E\left[ \sum_{\brho\in\IndSet}c_{\brho} A_{\brho}(t)\right]
\end{align}
which does not depend on the policy $\pi$ and is finite for all $\beta \in (0,1)$.
Hence, we obtain
$$(1-\beta)J_\beta(\bq,\pi)\; =\; c^\pi(0)+g-\beta \tilde{J}_\beta(\bq, \pi),$$
and thus any policy that minimizes $J_\beta$ also maximizes $\tilde{J}_\beta$.
\Halmos
\endproof

To solve problem \eqref{prob:opt_tilde_J}, we first express the associated Bellman equation as
\begin{equation*}
V(\bq)=\max_{\bs\in\cP}\left\{ r(\bq,\bs)+\beta\E[  V((\bq-\bs)^++\bA) ]  \right\}.
\end{equation*}
The optimal stationary policy is given by the maximizing schedule for each state $\bq$, which we solve using {\it value iteration}.
Let $V_{0}(\bq)=0$ for all $\bq\in\Ints_+^{4}$, and for the $(n+1)^{\text{th}}$ iteration, we define the \emph{$(n+1)^{\text{th}}$ value function} as
\begin{equation}
V_{n+1}(\bq):=\max_{\bs\in\cP}\left\{ r(\bq,\bs)+\beta\E[  V_n((\bq-\bs)^++\bA) ]  \right\}.
\label{eq:BellmanEq}
\end{equation}
In the next section, we state results on the properties of the value functions $V_n$ obtained using value iteration, as well as the schedules that achieve the maximum in the Bellman equation~\eqref{eq:BellmanEq}. 

\section{Optimal Scheduling}
\label{sec:optimal}
We now derive an optimal scheduling policy and related structural properties for our stochastic optimal control problem \eqref{prob:opt_J}
through the equivalent problem \eqref{prob:opt_tilde_J} above based on a reward for each period in terms of maximizing the number of packets served.
More specifically, we prove that an identified optimal policy solves the Bellman equation~\eqref{eq:BellmanEq} for any $n$, together with corresponding
structural properties, which implies that the policy renders solutions to both problems~\eqref{prob:opt_tilde_J} and \eqref{prob:opt_J}.
Our main results are established by solving \eqref{prob:opt_tilde_J} using value iteration over the decision space,
which we partition into three types of regions, namely the trivial boundary, the interior, and the critical boundary.
To summarize our results, the optimal scheduling policy coincides with the well known $c\mu$-rule in the interior;
while, in the trivial boundary, the optimal policy selects one of the schedules that can serve all the nonempty queues;
and, in the critical boundary, an optimal switching-curve policy will be followed, which can be reduced to the $c\mu$-rule in some cases. 
The proofs of our results are deferred until Section~\ref{sec:proofs}.

For notational convenience, let $\bfe_{\brho}$ represent the state in which only one packet exists in buffer $\brho$, for any $\brho\in\IndSet$, and all other buffers are empty.
%
We also write $\bmu\#\bnu$ when the two queues $\bmu,\bnu \in \IndSet$ cannot be contained in any schedule;
e.g., $(1,1)\#(1,2)$ and $(2,2)\#(2,1)$.

\subsection{Trivial Boundary}
\label{sec:trivial}
%
\begin{definition}
A state $\bq \in\mathbb Z_+^{4}$ is in the \emph{trivial boundary} if there exists $\bs\in \cP$ such that
\begin{equation}
  	q_{\brho}=0, \qquad \textrm{when $s_{\brho}=0$}. \label{eq:trivial_boundary}
\end{equation}
In other words, $\bs$ is a schedule that can serve packets in all nonempty queues in $\bq$.
\end{definition}
Our main result for the trivial boundary is expressed as follows.
\begin{theorem}\label{thm:trivial_boundary}
An optimal policy in every value iteration for $\bq$ in the trivial boundary is to choose a schedule $\bs\in\cP$ that satisfies \eqref{eq:trivial_boundary}.
\end{theorem}
This theorem can be derived from the following proposition.
\begin{proposition}\label{prop:trivial_boundary}
Any value function $V_n$ from the value iteration satisfies
\begin{equation}
	\beta V_n(\bq+\bfe_{\brho})\leq \beta V_n(\bq)+c_{\brho} , \label{eq:prop1:result}
\end{equation}
for any $\bq \in\Int_+^{4}$ and $\brho\in\IndSet$.
\end{proposition}

\subsection{Interior Region}
\label{sec:interior}
%
Define
\begin{equation*}
  r_{\max} \; \eqdef \; \max \{ r(\bq,\bs)\,:\,\bq\in\Int_+^{4},\, \bs\in \cP\}.
\end{equation*}
\begin{definition}
A state $\bq$ is an \emph{interior point} if 
\begin{eqnarray}
\max\{\, r(\bq,\bs)\,:\,\bs\in \cP\, \} \; = \; r_{\max},
\end{eqnarray}
and the \emph{interior region} comprises the set of all interior points.
\end{definition}

The following theorem identifies an optimal scheduling policy for the interior region,
rendering the $c\mu$ policy to be optimal.
\begin{theorem}\label{thm:interior}
An optimal schedule in any value iteration on an interior point $\bq$ is a schedule $\bs\in\cP$ such that $r(\bq,\bs)=r_{\max}$.
\end{theorem}
A crucial fact, which will be a key step for proving the theorem, is the following inequality that the value iteration function $V_n$ satisfies.
\begin{proposition}\label{prop:interior}
Let $\bq\in\Int_+^{4}$ be an interior point and $\bs\in\cP$ a schedule such that $r(\bq,\bs)=r_{\max}$.
Then, for any value function
in the value iteration,
and any schedule $\bs^\prime\in\cP$ with $\bs^\prime\leq \bq$, we have
\begin{align}
	r(\bq,\bs)+\beta V_n(\bq-\bs)\geq r(\bq,\bs^\prime)+\beta V_n( \bq-\bs^\prime ).
\label{eq:interior}
\end{align}	
\end{proposition}

\subsection{Critical Boundary}

We refer to the remaining region of the decision space as the critical boundary, and discuss two different cases for the optimal policy.

\subsubsection{Critical Boundary I: When \boldmath{\normalsize $c\mu$} is Optimal.}
\label{sec:critical-cmu}
We start by considering the case where only one buffer is empty and the $c\mu$ policy is optimal, as in the interior region.
%
\begin{theorem}\label{thm:three_queue}
Let $\IndSet=\{\bmu,\bnu,\brho,\bomega\}$ where $\bmu\#\bomega$ and $\bmu\#\brho$.
Further assume $c_{\bmu}\leq c_{\brho}+c_{\bomega}\leq c_{\bmu}+c_{\bnu}$, 
and let the state $\bq$ be such that $q_{\bnu}=0$, with all other queues nonempty.
Then, the optimal action on state $\bq$ is to serve packets in queues $\brho$ and $\bomega$ in any value iteration.
\end{theorem}
The above statement follows from the following proposition on the value function $V_n$.
\begin{proposition}\label{prop:three_queue}
Let $\IndSet=\{\bmu,\bnu,\brho,\bomega\}$ where $\bmu\#\bomega$ and $\bmu\#\brho$.
Assume that $c_{\bmu}\leq c_{\brho}+c_{\bomega}\leq c_{\bmu}+c_{\bnu}$. 
Then, for any
value function $V_n$ from the value iteration, we have
\begin{equation}
c_{\brho}+c_{\bomega}+\beta V_n(\bq+\bfe_{\bmu}) \geq c_{\bmu}+\beta V_n(\bq+\bfe_{\brho}+\bfe_{\bomega})
\label{eq:three_queue}
\end{equation}
for any $\bq\in\Int_+^{4}$.
\end{proposition}

\subsubsection{Critical Boundary II: When Switching Curve is Optimal.}
\label{sec:critical-switchcurve}
%
Now we consider the remainder of the critical boundary cases and show that an optimal policy of any value function has a switching curve structure.
This switching curve structure is characterized in Theorem~\ref{thm:switching_curve}, which defines regions of optimal actions that depend upon the state of the system and that are based on the corresponding value function inequalities in Proposition~\ref{PROP:SWITCHING_CURVE}.
\begin{theorem}\label{thm:switching_curve}
Fix a state $\bq\in\Int_+^{4}$ and $\bmu,\bnu\in\IndSet$.
In any value iteration, if an optimal action on $\bq$ is to serve queues $\bmu$ and $\bnu$ simultaneously,
then this is an optimal action on $\bq+\bfe_{\bmu}$ and $\bq+\bfe_{\bnu}$.
Therefore, in that value iteration, an optimal action on $\bq^\prime$ is to serve queues $\bmu$ and $\bnu$ if $q^\prime_{\bmu} \geq q_{\bmu}$,
$q^\prime_{\bnu}\geq q_{\bnu}$, and $q^\prime_{\brho}= q_{\brho}$ for all $\brho\in\IndSet$ such that $\brho\#\bmu$.
\end{theorem}

\begin{remark}\label{rem:switching_curve}
While this theorem is applicable for any $\bq$, the results of Theorem~\ref{thm:switching_curve} simply coincide with the above results for the interior region in Section~\ref{sec:interior}, the trivial boundary in Section~\ref{sec:trivial}, and the critical boundary in Section~\ref{sec:critical-cmu}
under the corresponding conditions.
Our use of Theorem~\ref{thm:switching_curve} in this section is to establish the optimal switching curve structure for the critical boundary when the conditions of
Section~\ref{sec:critical-cmu}
do not hold.
\end{remark}

To establish Theorem~\ref{thm:switching_curve} on a switching curve structure for the relevant portion of the critical boundary,
\begin{proposition}\label{PROP:SWITCHING_CURVE}
For every $n\in \Int_+$, the $n$-th value function satisfies the following inequalities:
For any $\bq\in\Int_+^{4}$,
\begin{align}
V_n(\bq+\bfe_{\bmu}+\bfe_{\brho})+V_n(\bq+\bfe_{\bmu}) & \geq V_n(\bq+2\bfe_{\bmu})+V_n(\bq+\bfe_{\brho}), \label{eq:two_queues}\\
V_n(\bq+\bfe_{\bmu}+\bfe_{\brho})+V_n(\bq+\bfe_{\bmu}+\bfe_{\bnu}) & \geq V_n(\bq+2\bfe_{\bmu}+\bfe_{\bnu})+V_n(\bq+\bfe_{\brho}), \label{eq:three_queues_other}\\
V_n(\bq+\bfe_{\bmu}+\bfe_{\brho}+\bfe_{\bomega})+V_n(\bq+\bfe_{\bmu}) & \geq V_n(\bq+2\bfe_{\bmu})+V_n(\bq+\bfe_{\brho}+\bfe_{\bomega}), \label{eq:three_queues_same}
\end{align}
where $\bmu,\brho,\bomega\in\IndSet$, $\bmu\# \brho$, $\bmu\# \bomega$, and $\brho\neq \bomega$.
\end{proposition}

\subsection{Identifying the Optimal Policy}

Theorems~\ref{thm:interior}~--~\ref{thm:switching_curve}
establish that an optimal scheduling policy follows the $c\mu$ rule in the interior region and in the trivial boundary while generally having a switching curve structure
in the critical boundary.
Hence, upon identifying the switching curve for the critical boundary, we have complete information about our optimal scheduling policy.
In the following, we first identify the precise switching curve in the special case of symmetric arrivals and unit costs, and then we propose an approximation algorithm
for the general case which is shown to be asymptotically optimal.

\subsubsection{Symmetric Arrivals and Unit Costs.}
\label{sec:optimal:symmetric}
%
Assume that $c_{\brho}=1$ and the arrival processes have the same rate
$\lambda_{\brho}=\lambda$, for all $\brho\in\IndSet$ and all $t$.
In this case, we further assume that $\lambda < 1/2$ to ensure the load is within the capacity region $\cC$;
we also have from \eqref{eq:g-constant} in this case that $g=4\lambda\sum_{t=0}^{\infty} \beta^{t+1}=4\lambda\beta/(1-\beta)$.

For this symmetric case, the interior region comprises all states in which the queues $(1,1)$ and $(2,2)$ or the queues $(1,2)$ and $(2,1)$ are nonempty
(i.e., the states in which the system can transmit two packets), whereas the trivial boundary comprises states with only one nonempty queue.
The critical boundary consists of the states in which there are two nonempty queues but only one packet can be transmitted.
We then have the following explicit characterization of a scheduling algorithm that we then prove to be optimal.

{\bf Algorithm~1.}~~{\it For the $2 \times 2$ input-queued switch with symmetric arrivals and unit costs, we define the \emph{size} of a schedule $\bs \in \cP$
to be the number of non-empty queues included in that schedule.
In every slot $t$, a schedule is then chosen in the following order:
\begin{enumerate}
\item[(i)] Select a size-$2$ schedule, if it exists, with ties broken according to an arbitrary well-defined (possibly randomized) rule;
\item[(ii)] Otherwise, if there are multiple (two) size-$1$ schedules, then the longest queue among them is served;
\item[(iii)] Otherwise, the queue of the unique size-$1$ schedule is served.
\end{enumerate}
}

To prove the optimality of Algorithm~1, we need the following proposition, which uses the i.i.d. assumption on arrivals.
\begin{proposition}\label{PROP:SYMMETRIC}
Any value function $V_n$ from the value iteration satisfies
\begin{equation}\label{eq:symmetric}
V_n(x\,\bfe_{\brho}+y\,\bfe_{\bomega}+z\,\bfe_{\bmu}+w\,\bfe_{\bnu})=V_n(z\,\bfe_{\brho}+w\,\bfe_{\bomega}+x\,\bfe_{\bmu}+y\,\bfe_{\bnu}),
\end{equation}	
where $\IndSet=\{\brho,\bomega,\bmu,\bnu\}$ with $\brho\#\bmu$, $\brho\#\bnu$ and $(x,y,z,w)\in\Ints_+^4$.
\end{proposition}

The above proposition, together with Theorems~\ref{thm:interior}, \ref{thm:trivial_boundary} and \ref{thm:switching_curve},
is shown in Theorem~\ref{thm:optimal} to identify the optimal actions in any value iteration.
%

%

\begin{theorem}\label{thm:optimal}
For $2\times 2$ input-queued switches with symmetric arrivals and unit costs, Algorithm~1 is optimal and
minimizes the discounted infinite horizon cost $J_{\beta}(\bq,\pi)$ for any $\beta \in (0,1)$.
\end{theorem}

\begin{remark}\label{rem:AC-MDP}
It is well known that a stationary optimal policy for the discounted-cost MDP with discount factor $\beta$ tending to $1$ can be used to obtain a stationary optimal policy
for the corresponding average-cost MDP;
see, e.g., \cite[Chapter 5]{O-L-book}, \cite[Chapter 7]{sennott-book}.
Therefore, Algorithm~$1$ can also be used in a similar manner to obtain an average-cost optimal policy.
\end{remark}

%

\subsubsection{General Case.}
\label{sec:optimal:general}
In contrast to the case of unit costs and symmetric arrivals of the previous section, deriving an explicit switching curve for our optimal scheduling policy is difficult
in general as it represents the solution to a general multidimensional stochastic optimal control problem;
and, in particular, the structure of the switching curve for our optimal policy can depend on the arrival processes in addition to other aspects of the optimal control problem.
Hence, instead of an explicit optimal solution, we investigate a ``look ahead'' policy based on value iterations,
which we show to be asymptotically optimal with respect to the degree of look ahead.
We note that general background on and analysis of look-ahead policies can be found in textbooks on stochastic control; see, e.g., \cite{Bert12,Pute05}.
Throughout this section, let $\cV$ denote the set of bounded real-valued functions on the
state space $\Int_+^{|\IndSet|} \simeq \Int_+^4$ with supremum norm 
$\norm{V}\eqdef\sup \{ |V(\bq)|\,:\,\bq\in\Int_+^{4}\}$, $V\in \cV$.
We also define $V^*_\beta,\tilde{V}^*_\beta\in \cV$ by
\begin{equation*}
V^*_\beta(\bq)\eqdef\max\{ J_\beta(\bq,\pi)\,:\,\pi\in\cM \}, \qquad\qquad \tilde{V}^*_\beta(\bq)\eqdef\max\{ \tilde{J}_\beta(\bq,\pi)\,:\,\pi\in\cM \},
\end{equation*}
recalling $\cM$ to be the set of all stationary Markov policies.

%
Consider, as in Section~\ref{sec:alt_formulation}, value iteration on the optimization problem \eqref{prob:opt_tilde_J} starting with $V_0=0$,
which can be viewed as solving the optimization problem over the look-ahead horizon with future values ignored beyond the horizon.
More specifically, we define the $\ell$-th look-ahead policy $\pi_\ell$ to be the policy that exploits the $\ell$-th value function as an approximation of an optimal solution,
thus yielding
\begin{equation}
\pi_\ell(\bq)\eqdef\argmax\left\{ r(\bq,\bs)+\E\left[ V_\ell\left( (\bq-\bs)^++\bA \right) \right]\,:\,\bs\in\cP \right\}.
\label{look-ahead.policy}
\end{equation}

This class of look-ahead policies has several important benefits, two of which we briefly highlight based on our theoretical results.
\begin{enumerate}
\item[(i).]  Our optimal results for the interior and trivial boundary can be exploited to significantly reduce the computational burden of the look-ahead policy.
Note that policy $\pi_k$ is the same policy that generates the $(k+1)$-th value function.
Since the optimal actions on states in the interior and the trivial boundary are known, we only need to determine the optimal actions for states in the critical boundary.
\item[(ii).]  For sufficiently large $\ell$, we can establish that policy $\pi_\ell$ is a good approximation to an optimal solution of problem \eqref{prob:opt_J}. 
Since $\pi_\ell$ is based on value iterations to solve \eqref{prob:opt_tilde_J}, then $\tilde{J}_\beta(\bq,\pi_\ell)$ is an approximation to $\tilde{V}_\beta^*(\bq)$.
Furthermore, in the following theorem, we prove that $J_\beta(\bq,\pi_\ell)$ converges to $V_\beta^*(\bq)$ as
$\ell \rightarrow \infty$.
\end{enumerate}

\begin{theorem}\label{thm:convergence} 
Let $V_0=0$ and let $\pi_\ell$ be the look-ahead policy produced by value iteration for $\ell=1,2,\cdots$.
Then, $J_\beta(\,\cdot\,,\pi_\ell)$ converges to $V^*_\beta(\,\cdot\,)$ as $\ell\to\infty$.
More precisely, if the inequality
\begin{equation}
	\norm{V^{\ell+1}-V^\ell}<\frac{\varepsilon(1-\beta)^2}{2\beta^2} \label{eq:conv_assume}
\end{equation}
holds for some $\varepsilon>0$, then we can conclude that $\norm{J_\beta(\,\cdot\,,\pi)-V_\beta^*}<\varepsilon$.
\end{theorem}
%

\section{Proofs of Main Results}
\label{sec:proofs}
%

In this section we turn to the proofs of
our main results from the previous section, with some additional proofs of technical results provided in the appendix.
We start with the trivial boundary, because some of these results are used for other regions, and then consider the interior region and critical boundary.

\subsection{Trivial Boundary}
\subsubsection{Proof of Proposition~\ref{prop:trivial_boundary}.} We show that the value function $V_n$ satisfies
\begin{equation*}
    \beta V_n(\bq+\bfe_{\brho})\leq \beta V_n(\bq)+c_{\brho}, \tag{\ref{eq:prop1:result} Revisited}
\end{equation*}
for all $\bq \in\Int_+^{4}$ and $\brho\in\IndSet$ by induction on $n$.
First, since $V_{0}(\bq)=0$ for any $\bq\in\Int_+^{4}$, \eqref{eq:prop1:result} holds for $n=0$.
Next, suppose that $V_k$ satisfies \eqref{eq:prop1:result}, let $\bs\in \cP$ be a schedule, and consider two cases.
\begin{enumerate}
\item[(i).] If ($s_{\brho}=0$) or ($s_{\brho}=1$ and $q_{\brho} \geq 1$), we have 
\begin{equation*}
r(\bq+\bfe_{\brho},\bs)=r(\bq,\bs), \quad \beta V_k((\bq+\bfe_{\brho}-\bs)^++\bA)=\beta V_k((\bq-\bs)^++\bA+\bfe_{\brho}) \leq \beta V_k((\bq-\bs)^++\bA)+c_{\brho},
\end{equation*}
where the inequality follows from the induction hypothesis.
\item[(ii).] Otherwise (i.e., $s_{\brho}=1$ and $q_{\brho}=0$), we obtain
\begin{equation*}
r(\bq+\bfe_{\brho},\bs)=r(\bq,\bs)+c_{\brho}, \qquad \beta V_k((\bq+\bfe_{\brho}-\bs)^++\bA)=\beta V_k((\bq-\bs)^++\bA).
\end{equation*}
\end{enumerate}

From (i) and (ii), we derive
\begin{align*}
\beta V_{k+1}(\bq+\bfe_{\brho}) =&\beta \max_{\bs\in\cP}\left\{ r(\bq+\bfe_{\brho},\bs)+\beta\E[V_k((\bq+\bfe_{\brho}-\bs)^++\bA)] \right\}\\
\leq& \beta \max_{\bs\in\cP}\left\{ r(\bq,\bs)+\beta\E[V_k((\bq-\bs)^++\bA)] \right\}+\beta c_{\brho}\\
=& \beta V_{k+1}(\bq)+\beta c_{\brho} \leq \beta V_{k+1}(\bq)+c_{\brho},
\end{align*}
which implies that $V_{k+1}$ satisfies \eqref{eq:prop1:result} and, by induction, the proof of Proposition~\ref{prop:trivial_boundary} is complete.

\subsubsection{Proof of Theorem~\ref{thm:trivial_boundary}.}
\label{proof:trivial}
Suppose that \eqref{eq:prop1:result} holds for $V_n$.
Let $\bq$ be a state in the trivial boundary and $\bs$ the schedule that satisfies \eqref{eq:trivial_boundary}.
Then, for any schedule $\bs^\prime\in\cP$, we have
\begin{equation*}
(\bq-\bs^\prime)^+=(\bq-\bs)^++\sum_{\brho\in\IndSet^\prime}\bfe_{\brho}, \qquad\qquad
    r(\bq,\bs^\prime)=r(\bq,\bs)-\sum_{\brho\in\IndSet^\prime} c_{\brho},
\end{equation*}
where $\IndSet^{\prime}=\{\brho\in\IndSet\,|\,q_{\brho}\geq 1\mbox{ and } s^{\prime}_{\brho}=0\}$.
Hence, we obtain
\begin{align*}
r(\bq,\bs^\prime)+\beta\E[V_n( (\bq-\bs^\prime)^++\bA)]
    =& r(\bq,\bs^\prime)+\beta\E\left[V_n\left( (\bq-\bs)^++\bA +\sum\nolimits_{\brho\in\IndSet^\prime}\bfe_{\brho} \right)\right]\\
    \leq& r(\bq,\bs^\prime)+\beta\E[V_n( (\bq-\bs)^++\bA )]+\sum\nolimits_{\brho\in\IndSet^\prime} c_{\brho} \\ 
    =& r(\bq,\bs)+\beta\E[V_n( (\bq-\bs)^++\bA )],
\end{align*}
where the inequality follows from Proposition~\ref{prop:trivial_boundary}.
As a result, $\bs$ is the optimal schedule for $\bq$ in any value iteration.

\subsection{Interior Region}
\label{proof:interior}
\subsubsection{Proof of Proposition~\ref{prop:interior}.}
Let $\bq\in\Int_+^{4}$ be an interior point and $\bs\in\cP$ a schedule such that $r(\bq,\bs)=r_{\max}$. If $\bs^\prime\leq \bs$, then \eqref{eq:interior} holding for $V_{k+1}$ immediately follows from Proposition~\ref{prop:trivial_boundary}.
We therefore
focus on the other case
which, in a $2 \times 2$ switch, 
means that $\bs^\prime$ and $\bs$ have no common queue.
Now we use
induction on $n$, and show that any value function satisfies
\begin{align*}
    r(\bq,\bs)+\beta V_n(\bq-\bs)\geq r(\bq,\bs^\prime)+\beta V_n( \bq-\bs^\prime ),\tag{\ref{eq:interior} Revisited}
\end{align*} 
for any schedule $\bs^\prime\in\cP$ with $\bs^\prime\leq \bq$. First, for $n=0$, \eqref{eq:interior} holds because $V_{0}(\bq)=0$ and $r(\bq,\bs)=r_{\max}\geq r(\bq,\bs^\prime)$ for any $\bs^\prime\in\cP$.
Next, assume that $V_k$ satisfies \eqref{eq:interior}.
%


Since $\bs^\prime\leq \bq$ and since $\bs^\prime$ and $\bs$ have no common queue, we have that  $\bq-\bs^\prime$ is an interior point with $r(\bq-\bs^\prime,\bs)=r_{\max}$. 
Hence, we obtain from the induction hypothesis that
\begin{equation*}
V_{k+1}(\bq-\bs^\prime)=r(\bq-\bs^\prime,\bs)+\beta\E[V_k(\bq-\bs-\bs^\prime+\bA)] =r(\bq,\bs)+\beta\E[V_k(\bq-\bs-\bs^\prime+\bA)].
\end{equation*}
Then, from the definition of the value iteration, we have
\begin{equation*}
V_{k+1}(\bq-\bs)\geq r(\bq-\bs,\bs^\prime)+\beta\E[V_k(\bq-\bs-\bs^\prime+\bA)] =r(\bq,\bs^\prime)+\beta\E[V_k(\bq-\bs-\bs^\prime+\bA)],
\end{equation*}
so that
\begin{align*}
	V_{k+1}(\bq-\bs^\prime)-r(\bq,\bs)\leq V_{k+1}(\bq-\bs)-r(\bq,\bs^\prime),
\end{align*}
which implies that \eqref{eq:interior} holds for $n=k+1$, since $\beta<1$. Thus, the proof of Proposition~\ref{prop:interior} is complete by induction.


\subsubsection{Proof of Theorem~\ref{thm:interior}.}
%
For any interior point $\bq$ with a schedule $\bs\in\cP$ such that $r(\bq,\bs)=r_{\max}$ and any $\bs^\prime\leq\bq$, we have
\begin{align*}
r(\bq,\bs)+\beta\E[V_n(\bq-\bs+\bA)] & = \E[ r(\bq+\bA,\bs)+\beta V_n(\bq+\bA-\bs)]\\
& \geq\E[ r(\bq+\bA,\bs^\prime)+\beta V_n(\bq+\bA-\bs^\prime)] =r(\bq,\bs^\prime)+\beta\E[V_n(\bq+\bA-\bs^\prime)],
\end{align*}
where the first and the last equalities follow from $\bq+\bA\geq\bs,\bs^\prime$ (which implies $r(\bq,\bs)=r(\bq+\bA,\bs)$), and the inequality follows from \eqref{eq:interior} for $V_n$.
Hence, Theorem~\ref{thm:interior} holds in any value iteration.

\subsection{Critical Boundary: When \boldmath{\normalsize $c\mu$} is Optimal.}

\subsubsection{Proof of Proposition~\ref{prop:three_queue}.}

\label{proof:critical-cmu}
%
Under the assumptions of the proposition, we prove by induction that, for any value function $V_n$,
\begin{equation*}
c_{\brho}+c_{\bomega}+\beta V_n(\bq+\bfe_{\bmu}) \geq c_{\bmu}+\beta V_n(\bq+\bfe_{\brho}+\bfe_{\bomega}),\tag{\ref{eq:three_queue} Revisited}
\end{equation*}
for any $\bq\in\Int_+^{4}$.
First, for $n=0$, $V_{0}$ satisfies \eqref{eq:three_queue} because $V_{0}(\bq)=0$ for any $\bq\in\Int_+^{4}$ and $c_{\bmu}\leq c_{\brho}+c_{\bomega}$.
Next, assume that \eqref{eq:three_queue} holds for $V_k$ and consider two cases.
\begin{enumerate}
\item[(i)] Suppose that 
\begin{equation*}
V_{k+1}(\bq+\bfe_{\brho}+\bfe_{\bomega})=c_{\bmu} \cdot \indic_{\{q_{\bmu}>0\}}+c_{\bnu} \cdot \indic_{\{q_{\bnu}>0\}}
	+\beta\E[V_k((\bq-\bfe_{\bmu}-\bfe_{\bnu})^++\bA+\bfe_{\brho}+\bfe_{\bomega})].
\end{equation*}
If $q_{\bmu}\geq 1$, we have
\begin{align*}
V_{k+1}(\bq+\bfe_{\brho}+\bfe_{\bomega})&=c_{\bmu}+c_{\bnu} \cdot \indic_{\{q_{\bnu}>0\}} +\beta\E[V_k((\bq-\bfe_{\bnu})^++\bA+\bfe_{\brho}+\bfe_{\bomega}-\bfe_{\bmu})]\\
&\leq c_{\brho}+c_{\bomega}+c_{\bnu} \cdot \indic_{\{q_{\bnu}>0\}}+\beta\E[V_k((\bq-\bfe_{\bnu})^++\bA)]
\leq c_{\brho}+c_{\bomega}+V_{k+1}(\bq+\bfe_{\bmu})-c_{\bmu},
\end{align*}
where the first inequality follows from the induction hypothesis and the second inequality follows from the definition of the value iteration.
On the other hand,
if $q_{\bmu}=0$, we obtain
\begin{align*}
V_{k+1}(\bq+\bfe_{\brho}+\bfe_{\bomega}) =&c_{\bnu} \cdot \indic_{\{q_{\bnu}>0\}}+\beta\E[V_k((\bq-\bfe_{\bnu})^++\bA+\bfe_{\brho}+\bfe_{\bomega})]\\
\leq& c_{\bnu} \cdot \indic_{\{q_{\bnu}>0\}}+c_{\brho}+c_{\bomega}+\beta\E[V_k((\bq-\bfe_{\bnu})^++\bA)]
\leq c_{\brho}+c_{\bomega}+V_{k+1}(\bq+\bfe_{\bmu})-c_{\bmu},
\end{align*}
where the first inequality follows from Proposition~\ref{prop:trivial_boundary} and the second inequality follows from the definition of the value iteration. This leads to 
\begin{align*}
c_{\brho}+c_{\bomega}+\beta V_n(\bq+\bfe_{\bmu}) \geq c_{\bmu}+\beta V_n(\bq+\bfe_{\brho}+\bfe_{\bomega}), 
\end{align*}
since $\beta<1$.
\item[(ii)] Otherwise, suppose that 
\begin{equation*}
V_{k+1}(\bq+\bfe_{\brho}+\bfe_{\bomega})=c_{\brho}+c_{\bomega}+\beta\E[V_k(\bq+\bA)].
\end{equation*}
If $q_{\bnu}\geq 1$, we have
\begin{align*}
V_{k+1}(\bq+\bfe_{\brho}+\bfe_{\bomega}) =&c_{\brho}+c_{\bomega}+\beta\E[V_k(\bq+\bA)]\\
\leq&c_{\brho}+c_{\bomega}+c_{\bnu}+\beta\E[V_k(\bq+\bA-\bfe_{\bnu})] \leq c_{\brho}+c_{\bomega}+V_{k+1}(\bq+\bfe_{\bmu})-c_{\bmu},
\end{align*}
where the first inequality follows from Proposition~\ref{prop:trivial_boundary} and the second inequality follows from the definition of the value iteration. 
However, if $q_{\bnu}=0$, we obtain
\begin{align*}
V_{k+1}(\bq+\bfe_{\brho}+\bfe_{\bomega}) & =c_{\brho}+c_{\bomega}+\beta\E[V_k(\bq+\bA)] \\
	& \leq c_{\brho}+c_{\bomega}+V_{k+1}(\bq+\bfe_{\bmu})-c_{\bmu},
\end{align*}
where the inequality follows from the definition of the value iteration. 
\end{enumerate}

Hence, \eqref{eq:three_queue} holds for $V_{k+1}$ since $\beta<1$ and, by induction, the proof of Proposition~\ref{prop:three_queue} is complete.

\subsubsection{Proof of Theorem~\ref{thm:three_queue}.}
Under the assumptions of the theorem, recall $\bq\in\Int_+^{4}$ to be a state such that $q_{\bnu}=0$ and all other queues are nonempty.
Then, possible schedules at state $\bq$ are $\bfe_{\bmu}$ and $\bfe_{\brho}+\bfe_{\bomega}$.
From Proposition~\ref{prop:three_queue}, we have
\begin{align*}
r(\bq,\bfe_{\bmu})+\beta\E[V_n(\bq-\bfe_{\bmu}+\bA)] & = \E\left[ c_{\bmu}+\beta V_n\left((\bq+\bA-\bfe_{\bmu}-\bfe_{\brho}-\bfe_{\bomega})+\bfe_{\brho}+\bfe_{\bomega}\right)\right]\\
& \leq\E\left[ c_{\brho}+c_{\bomega}+\beta V_n\left((\bq+\bA-\bfe_{\bmu}-\bfe_{\brho}-\bfe_{\bomega})+\bfe_{\bmu}\right)\right]\\
& =r(\bq,\bfe_{\brho}+\bfe_{\bomega})+\beta\E[V_n(\bq+\bA-(\bfe_{\brho}+\bfe_{\bomega}))],
\end{align*}
which implies Theorem~\ref{thm:three_queue} holds.

\subsection{Critical Boundary: When Switching Curve is Optimal.}
\label{proof:critical-switchcurve}
In this subsection, we prove Theorem~\ref{thm:switching_curve} in two steps.
First, in Section~\ref{subsec:switchingcurve}, we present a lemma establishing that the value function satisfies the conclusion of Theorem~\ref{thm:switching_curve} under the conditions \eqref{eq:two_queues}, \eqref{eq:three_queues_other} and \eqref{eq:three_queues_same} of Proposition~\ref{PROP:SWITCHING_CURVE}.
Then, in Section \ref{subsec:propswitchingcurve}, we prove Proposition~\ref{PROP:SWITCHING_CURVE} and show that the value function indeed satisfies
\eqref{eq:two_queues}, \eqref{eq:three_queues_other} and \eqref{eq:three_queues_same}, 
thus immediately yielding Theorem~\ref{thm:switching_curve}.

\subsubsection{Proof of Theorem~\ref{thm:switching_curve}.\label{subsec:switchingcurve}}
Our first step in the proof of Theorem~\ref{thm:switching_curve} is to establish the following lemma.

\begin{lemma}\label{lem:IH_switching_curve}
Let $V_n$ be the value function from the $n^{\textrm{th}}$ step of value iteration. 
Suppose that $V_n$ satisfies
\begin{align*}
V_n(\bq+\bfe_{\bmu}+\bfe_{\brho})+V_n(\bq+\bfe_{\bmu}) & \geq V_n(\bq+2\bfe_{\bmu})+V_n(\bq+\bfe_{\brho}), \tag{\ref{eq:two_queues} Revisited}\\
V_n(\bq+\bfe_{\bmu}+\bfe_{\brho})+V_n(\bq+\bfe_{\bmu}+\bfe_{\bnu}) & \geq V_n(\bq+2\bfe_{\bmu}+\bfe_{\bnu})+V_n(\bq+\bfe_{\brho}), \tag{\ref{eq:three_queues_other} Revisited}\\
V_n(\bq+\bfe_{\bmu}+\bfe_{\brho}+\bfe_{\bomega})+V_n(\bq+\bfe_{\bmu}) & \geq V_n(\bq+2\bfe_{\bmu})+V_n(\bq+\bfe_{\brho}+\bfe_{\bomega}). \tag{\ref{eq:three_queues_same} Revisited}
\end{align*}
Consider the optimization problem on the right-hand side of the Bellman equation \eqref{eq:BellmanEq} involving $V_n(\cdot)$. If an optimal action on $\bq$  
is to serve queues $\bmu$ and $\bnu$ simultaneously,
then this is also an optimal action on $\bq+\bfe_{\bmu}$ and $\bq+\bfe_{\bnu}$.
 Therefore, an optimal action on $\bq^\prime$ is to serve queues $\bmu$ and $\bnu$ if $q^\prime_{\bmu} \geq q_{\bmu}$,
$q^\prime_{\bnu}\geq q_{\bnu}$, and $q^\prime_{\brho}= q_{\brho}$ for all $\brho\in\IndSet$ such that $\brho\#\bmu$.
\end{lemma}

\proof{Proof}
Note that if $\bq$ is in the interior or trivial boundary, the conclusion of
Theorem~\ref{thm:switching_curve} holds from Theorems~\ref{thm:interior} and \ref{thm:trivial_boundary}. Hence, we assume that state $\bq$ is in the critical boundary, and from the hypothesis of the lemma, we know that serving queues $\bmu$ and $\bnu$ is an optimal action on $\bq$ in the $n$-th value iteration.  

By symmetry, $\bq$ falls into one of the three subregions:
\begin{enumerate}
\item[{\bf C1}:] $q_{\bmu}\geq 1$, $q_{\bnu}=0$, $q_{\brho}\geq 1$, and $q_{\bomega}= 0$;
\item[{\bf C2}:] $q_{\bmu}\geq 1$, $q_{\bnu}\geq 1$, $q_{\brho}\geq 1$, and $q_{\bomega}=0$;
\item[{\bf C3}:] $q_{\bmu}\geq 1$, $q_{\bnu}=0$, $q_{\brho}\geq 1$, and $q_{\bomega}\geq 1$;
\end{enumerate}
where $\brho\neq\bomega$ are queues that cannot be served with $\bmu$. We now prove the lemma for each of these three cases.
\proof{Proof for {\bf C1}.}
Since serving queues $\bmu$ and $\bnu$ is the optimal action on $\bq$ in the $n$-th value iteration,
we have
\begin{equation}
	c_{\bmu}+\beta\E[V_n(\bq+\bA-\bfe_{\bmu})]
	\geq c_{\brho}+\beta\E[V_n(\bq+\bA-\bfe_{\brho})] , \label{eq:C1_assump}
\end{equation}
and substituting $\bq+\bA-\bfe_{\bmu}-\bfe_{\brho}\geq\bzero$ for $\bq$ in \eqref{eq:two_queues} yields
\begin{align*}
V_n(\bq+\bA)+V_n(\bq+\bA-\bfe_{\brho}) \geq V_n(\bq+\bA+\bfe_{\bmu}-\bfe_{\brho})+V_n(\bq+\bA-\bfe_{\bmu}).
\end{align*}
Taking expectation of the above equation for $\bA$ and adding this to \eqref{eq:C1_assump}, we obtain 
\begin{equation*}
	c_{\bmu}+\beta\E[V_n(\bq+\bA)]
	\geq c_{\brho}+\beta\E[V_n(\bq+\bA+\bfe_{\bmu}-\bfe_{\brho})],
\end{equation*}
which implies that the optimal action on $\bq+\bfe_{\bmu}$ is to serve queues $\bmu$ and $\bnu$.

On the other hand, for $\bq+\bfe_{\bnu}$, we have
\begin{equation*}
c_{\brho}+\beta\E[V_n(\bq+\bA+\bfe_{\bnu}-\bfe_{\brho})] \leq c_{\brho}+c_{\bnu}+\beta\E[V_n(\bq+\bA-\bfe_{\brho})] \leq c_{\bmu}+c_{\bnu}+\beta\E[V_n(\bq+\bA-\bfe_{\bmu})],
\end{equation*}
where the first inequality follows from Proposition~\ref{prop:trivial_boundary} and the second inequality follows from \eqref{eq:C1_assump}.
Hence, the optimal action on $\bq+\bfe_{\bnu}$ is to serve queues $\bmu$ and $\bnu$.
\Halmos
\endproof

\proof{Proof for {\bf C2}.}
Since serving queues $\bmu$ and $\bnu$ is the optimal action on $\bq$ in the $n$-th value iteration, we have from the definition of {\bf C2} that 
\begin{align}\label{eq:C2_assump}
\begin{array}{ll}
c_{\bmu}+c_{\bnu}+\beta\E[V_n(\bq+\bA-\bfe_{\bmu}-\bfe_{\bnu})] \geq c_{\brho}+\beta\E[V_n(\bq+\bA-\bfe_{\brho})].
\end{array}
\end{align}
Substituting $\bq+\bA-\bfe_{\bmu}-\bfe_{\bnu}-\bfe_{\brho}\geq\bzero$ for $\bq$ in \eqref{eq:three_queues_other} yields
\begin{align*}
	V_n(\bq+\bA-\bfe_{\bnu})+V_n(\bq+\bA-\bfe_{\brho}) \geq V_n(\bq+\bA+\bfe_{\bmu}-\bfe_{\brho})+V_n(\bq+\bA-\bfe_{\bmu}-\bfe_{\bnu}).
\end{align*}
Taking expectation of the above equation for $\bA$ and adding this to \eqref{eq:C2_assump}, we obtain 
\begin{equation*}
c_{\bmu}+c_{\bnu}+\beta\E[V_n(\bq+\bA-\bfe_{\bnu})] \geq c_{\brho}+\beta\E[V_n(\bq+\bA+\bfe_{\bmu}-\bfe_{\brho})],
\end{equation*}
which implies that the optimal action on $\bq+\bfe_{\bmu}$ is to serve packets in queues $\bmu$ and $\bnu$
and, by symmetry, this is also the optimal action on $\bq+\bfe_{\bnu}$. 
\Halmos
\endproof
\endproof

\proof{Proof for {\bf C3}.}
Since serving queues $\bmu$ and $\bnu$ is the optimal action on $\bq$ in the $n$-th value iteration,
we have
\begin{align}\label{eq:C3_assump}
\begin{array}{ll}
c_{\bmu}+\beta\E[V_n(\bq+\bA-\bfe_{\bmu})] 
	& \geq c_{\brho}+c_{\bomega}+\beta\E[V_n(\bq+\bA-\bfe_{\brho}-\bfe_{\bomega})] ,
\end{array}
\end{align}
and substituting $\bq+\bA-\bfe_{\bmu}-\bfe_{\brho}-\bfe_{\bomega}\geq\bzero$ for $\bq$ in \eqref{eq:three_queues_same} yields
\begin{align*}
V_n(\bq+\bA)+V_n(\bq+\bA-\bfe_{\brho}-\bfe_{\bomega})
		& \geq V_n(\bq+\bA+\bfe_{\bmu}-\bfe_{\brho}-\bfe_{\bomega})+V_n(\bq+\bA-\bfe_{\bmu}).
\end{align*}
Taking expectation of the above equation for $\bA$ and adding this to \eqref{eq:C3_assump}, we obtain 
\begin{equation*}
c_{\bmu}+\beta\E[V_n(\bq+\bA)] \geq c_{\brho}+c_{\bomega}+\beta\E[V_n(\bq+\bA+\bfe_{\bmu}-\bfe_{\brho}-\bfe_{\bomega})],
\end{equation*}
which implies that the optimal action on $\bq+\bfe_{\bmu}$ is to serve queues $\bmu$ and $\bnu$.

On the other hand, for $\bq+\bfe_{\bnu}$, we have
\begin{align*}
c_{\brho}+c_{\bomega}+\beta\E[V_n(\bq+\bA+\bfe_{\bnu}-\bfe_{\brho}-\bfe_{\bomega})]
	&\leq c_{\brho}+c_{\bnu}+c_{\bomega}+\beta\E[V_n(\bq+\bA-\bfe_{\brho}-\bfe_{\bomega})]\\
	&\leq c_{\bnu}+c_{\bmu}+\beta\E[V_n(\bq+\bA-\bfe_{\bmu})],
\end{align*}
where the first inequality follows from Proposition~\ref{prop:trivial_boundary} and the second inequality follows from \eqref{eq:C3_assump}.
Hence, the optimal action on $\bq+\bfe_{\bnu}$ is to serve packets in queues $\bmu$ and $\bnu$.
\Halmos
\endproof

Once Proposition~\ref{PROP:SWITCHING_CURVE} is established, Theorem~\ref{thm:switching_curve} follows from 
Lemma~\ref{lem:IH_switching_curve} above. 

\subsubsection{Proof of Proposition~\ref{PROP:SWITCHING_CURVE}. \label{subsec:propswitchingcurve}}
We first present a key lemma on inequality properties of the value function, followed by a proposition that subsumes Proposition~\ref{PROP:SWITCHING_CURVE}.

\begin{lemma}\label{LEMMA:AUX_INEQUALITY_ALT}
Suppose that $\bx,\by\in\{0,1\}^4$ and $\bz,\bw\in \Ints^{4}$ satisfy:
\emph{(a)} $\bx\leq \bfe_{\bmu}+\bfe_{\bnu}$ and $\by\leq\bfe_{\brho}+\bfe_{\bomega}$, component-wise;
and
\emph{(b)} $\bx+\by=\bz+\bw$.
Then, any value function $V_n$ satisfies
\begin{align}
	V_n(\bq+\bx)+V_n(\bq+\by)&\geq V_n(\bq+\bz)+V_n(\bq+\bw),\label{eq:aux_inequality}
\end{align}
for any $\bq\in\Int_+^{4}$.
\end{lemma}
This lemma, whose proof is provided in
Appendix~\ref{sec:pf:aux_ineqaulity_alt},
supports the proof of the following proposition, which trivially implies Proposition~\ref{PROP:SWITCHING_CURVE} and thus also completes the proof of Theorem~\ref{thm:switching_curve}.

\begin{proposition}\label{Prop:7}
For every $n\in \Int_+$, the $n$-th value function $V_n(\cdot)$ satisfies \eqref{eq:two_queues}~--~\eqref{eq:three_queues_same} and the following inequalities, for any $\bq\in\Int_+^{4}$:
\begin{align}
2 V_n(\bq+\bfe_{\bmu}) & \geq V_n(\bq)+V_n(\bq+2\bfe_{\bmu}), \label{eq:for_two}\\
V_n(\bq+\bfe_{\bmu}+\bfe_{\bnu})+V_n(\bq+\bfe_{\bmu}) &\geq  V_n(\bq+2\bfe_{\bmu}+\bfe_{\bnu})+V_n(\bq), \label{eq:for_three_different}
\end{align}
where $\bmu,\bnu,\brho,\bomega\in\IndSet$ such that $\brho\#\bmu$, $\bomega\#\bmu$, $\brho\neq\bomega$, $\bmu\neq\bnu$.
\end{proposition}
%


\proof{Proof}
We now prove the proposition by induction.
First, for $n=0$, all equations \eqref{eq:two_queues}~--~\eqref{eq:three_queues_same} and \eqref{eq:for_two}~--~\eqref{eq:for_three_different} hold
because $V_{0}(\bq)=0$ for all $\bq\in\Int_+^{4}$.
Next, assuming that the $k$-th value function satisfies all of these equations, we prove below that $V_{k+1}$ satisfies the
first three equations \eqref{eq:two_queues}~--~\eqref{eq:three_queues_same} and prove
in Appendices~\ref{sub:proof_of_eq:for_two_for_} and~\ref{app:for_three_different}
that $V_{k+1}$ satisfies the remaining equations.

%
\textbf{Proof of \eqref{eq:two_queues} for $V_{k+1}$.}
We prove \eqref{eq:two_queues} of Proposition~\ref{PROP:SWITCHING_CURVE} for $n=k+1$.
The right-hand side of this equation involves $V_{k+1}(\bq+2\bfe_{\bmu})$ and  $V_{k+1}(\bq+\bfe_{\brho})$. By the definition of $V_{k+1}(\cdot)$ in \eqref{eq:BellmanEq}, we need to work with the maximization problem on the right-hand side of \eqref{eq:BellmanEq} for $V_k(\cdot)$ at  $\bq+2\bfe_{\bmu}$ and at $\bq+\bfe_{\brho}$. We consider four cases based on the optimal schedules at $\bq+2\bfe_{\bmu}$ and at $\bq+\bfe_{\brho}$:
(1) Both optimal actions are to serve queues $\bmu$ and $\bnu$;
(2) Both optimal actions are to serve queues $\brho$ and $\bomega$;
(3) The optimal action on $(\bq+2\bfe_{\bmu})$ is to serve queues $\bmu$ and $\bnu$, and the optimal action on $\bq+\bfe_{\brho}$ is to serve queues $\brho$ and $\bomega$;
(4) The optimal action on $(\bq+2\bfe_{\bmu})$ is to serve queues $\brho$ and $\bomega$, and the optimal action on $\bq+\bfe_{\brho}$ is to serve queues $\bmu$ and $\bnu$.
We now prove \eqref{eq:two_queues} for the $(k+1)$-th value function dealing with all four cases.

First, suppose that optimal actions on $\bq+2\bfe_{\bmu}$ and $\bq+\bfe_{\brho}$ are to transmit packets in queues $\bmu$ and $\bnu$ in the $(k+1)$-th value iteration.
If $q_{\bmu}\geq 1$, we obtain
\begin{align*}
V_{k+1}(\bq+2\bfe_{\bmu})+V_{k+1}(\bq+\bfe_{\brho}) &\\
& \hspace*{-2.0in} = c_{\bmu}+c_{\bnu}\cdot \indic_{\{q_{\bnu}>0\}}+\beta\E[ V_k((\bq-\bfe_{\bnu})^++\bA+\bfe_{\bmu})]
+ c_{\bmu}+c_{\bnu}\cdot \indic_{\{q_{\bnu}>0\}}+\beta\E[ V_k((\bq-\bfe_{\bnu})^++\bA-\bfe_{\bmu}+\bfe_{\brho})]\\
& \hspace*{-2.0in} \leq c_{\bmu}+c_{\bnu}\cdot \indic_{\{q_{\bnu}>0\}}+\beta\E[ V_k( (\bq-\bfe_{\bnu})^++\bA+\bfe_{\brho} ) ]
+ c_{\bmu}+c_{\bnu}\cdot \indic_{\{q_{\bnu}>0\}}+\beta\E[ V_k((\bq-\bfe_{\bnu})^++\bA)]\\
& \hspace*{-2.0in} \leq V_{k+1}(\bq+\bfe_{\bmu}+\bfe_{\brho})+V_{k+1}(\bq+\bfe_{\bmu}) ;
\end{align*}
here the first inequality follows from the induction hypothesis (substituting $(\bq-\bfe_{\bnu})^++\bA-\bfe_{\bmu}$ for $\bq$ in \eqref{eq:two_queues} for $V_k$)
and the second inequality follows from the definition of the value iteration.
On the other hand, if $q_{\bmu}=0$, we have
\begin{align*}
V_{k+1}(\bq+2\bfe_{\bmu})+V_{k+1}(\bq+\bfe_{\brho}) &\\
& \hspace*{-1.5in} =c_{\bmu}+c_{\bnu}\cdot \indic_{\{q_{\bnu}>0\}}+\beta\E[ V_k((\bq-\bfe_{\bnu})^++\bA+\bfe_{\bmu})]
+c_{\bnu}\cdot \indic_{\{q_{\bnu}>0\}}+\beta\E[ V_k((\bq-\bfe_{\bnu})^++\bA+\bfe_{\brho})]\\
& \hspace*{-1.5in} \leq c_{\bmu}+c_{\bnu}\cdot \indic_{\{q_{\nu}>0\}}+\beta\E[  V_k((\bq-\bfe_{\bnu})^++\bA)]
+ c_{\bmu}+c_{\bnu}\cdot \indic_{\{q_{\nu}>0\}}+\beta\E[ V_k((\bq-\bfe_{\bnu})^++\bA+\bfe_{\brho})]\\
& \hspace*{-1.5in} \leq V_{k+1}(\bq+\bfe_{\bmu}+\bfe_{\brho})+V_{k+1}(\bq+\bfe_{\bmu}),
\end{align*}
where the first inequality follows from Proposition~\ref{prop:trivial_boundary} and the second inequality follows from the definition of the value iteration.

Second, assume that optimal actions on $\bq+2\bfe_{\bmu}$ and $\bq+\bfe_{\brho}$ are to transmit packets in queues $\brho$ and $\bomega$ in the $(k+1)$-th value iteration.
If $q_{\brho}\geq 1$, we obtain
\begin{align*}
V_{k+1}(\bq+2\bfe_{\bmu})+V_{k+1}(\bq+\bfe_{\brho}) &\\
& \hspace*{-2.0in} =c_{\brho}+c_{\bomega}\cdot \indic_{\{q_{\bomega}>0\}}+\beta\E[ V_k((\bq-\bfe_{\bomega})^++\bA-\bfe_{\brho}+2\bfe_{\bmu})]
+ c_{\brho}+c_{\bomega}\cdot \indic_{\{q_{\bomega}>0\}}+\beta\E[ V_k( (\bq-\bfe_{\bomega})^++\bA )]\\
& \hspace*{-2.0in} \leq c_{\brho}+c_{\bomega}\cdot \indic_{\{q_{21}>0\}}+\beta\E[ V_k((\bq-\bfe_{\bomega})^++\bA+\bfe_{\bmu})]
+ c_{\brho}+c_{\bomega}\cdot \indic_{\{q_{21}>0\}}+\beta\E[ V_k((\bq-\bfe_{\bomega})^++\bA-\bfe_{\brho}+\bfe_{\bmu})]\\
& \hspace*{-2.0in} \leq V_{k+1}(\bq+\bfe_{\bmu}+\bfe_{\brho})+V_{k+1}(\bq+\bfe_{\bmu}) ;
\end{align*}
here the first inequality follows from the induction hypothesis (substituting $(\bq-\bfe_{\bomega})^++\bA-\bfe_{\brho}$ for $\bq$ in \eqref{eq:two_queues}
for $V_k$) and the second inequality follows from the definition of the value iteration.
On the other hand, if $q_{\brho}=0$, we have
\begin{align*}
V_{k+1}(\bq+2\bfe_{\bmu})+V_{k+1}(\bq+\bfe_{\brho}) &\\
& \hspace*{-1.5in} =c_{\bomega}\cdot \indic_{\{q_{\omega}>0\}}+\beta\E[ V_k((\bq-\bfe_{\bomega})^++\bA+2\bfe_{\bmu})]
+c_{\brho}+c_{\bomega}\cdot \indic_{\{q_{\omega}>0\}}+\beta\E[ V_k(\bq-\bfe_{\bomega})^++\bA)]\\
& \hspace*{-1.5in} \leq c_{\brho}+c_{\bomega}\cdot \indic_{\{q_{\omega}>0\}}+\beta\E[ V_k( (\bq-\bfe_{\bomega})^++\bA+\bfe_{\bmu} )]
+ c_{\bomega}\cdot \indic_{\{q_{\omega}>0\}}+\beta\E[ V_k((\bq-\bfe_{\bomega})^++\bA+\bfe_{\bmu})]\\
& \hspace*{-1.5in} \leq V_{k+1}(\bq+\bfe_{\bmu}+\bfe_{\brho})+V_{k+1}(\bq+\bfe_{\bmu}),
\end{align*}
where the first inequality follows from the induction hypothesis (substituting $(\bq-\bfe_{\bomega})^++\bA$ for $\bq$ in \eqref{eq:for_two} for $V_k$)
and the second inequality follows from the definition of the value iteration.

Third, suppose that the optimal action on $\bq+2\bfe_{\bmu}$ and $\bq+\bfe_{\brho}$ is to serve packets in queues $\bmu$ and $\bnu$
and the optimal action on $\bq+\bfe_{\brho}$ is to transmit packets in queues $\brho$ and $\bomega$ in the $(k+1)$-th value iteration.
Then, we obtain
\begin{align*}
V_{k+1}(\bq+2\bfe_{\bmu})+V_{k+1}(\bq+\bfe_{\brho}) =& c_{\bmu}+c_{\bnu}\cdot \indic_{\{q_{\bnu}>0\}}+\beta\E[ V_k((\bq-\bfe_{\bnu})^++\bA+\bfe_{\bmu})]\\
&\quad + c_{\brho}+c_{\bomega}\cdot \indic_{\{q_{\bomega}>0\}}+\beta\E[ V_k((\bq-\bfe_{\bomega})^++\bA)]\\
=& c_{\bmu}+c_{\bnu}\cdot \indic_{\{q_{\bnu}>0\}} + c_{\brho}+c_{\bomega}\indic_{\{q_{\bomega}>0\}} + \beta\E[V_{k}(\bQ+\bz)+V_{k}(\bQ+\bw)] ,\\
V_{k+1}(\bq+\bfe_{\bmu})+V_{k+1}(\bq+\bfe_{\bmu}+\bfe_{\brho}) \geq& c_{\bmu}+c_{\bnu}\cdot \indic_{\{q_{\bnu}>0\}}+\beta\E[ V_k((\bq-\bfe_{\bnu})^++\bA)]\\
&\quad + c_{\brho}+c_{\bomega}\cdot \indic_{\{q_{\bomega}>0\}}+\beta\E[ V_k((\bq-\bfe_{\bomega})^++\bA+\bfe_{\bmu})]\\
=& c_{\bmu}+c_{\bnu}\cdot \indic_{\{q_{\bnu}>0\}} + c_{\brho}+c_{\bomega}\indic_{\{q_{\bomega}>0\}} + \beta\E[V_{k}(\bQ+\bx)+V_{k}(\bQ+\by)],
\end{align*}
where
$\bQ\eqdef (\bq-\bfe_{\bnu}-\bfe_{\bomega})^++\bA$,
$\bx\eqdef \bq-(\bq-\bfe_{\bnu})^++\bfe_{\bmu}$,
$\by\eqdef \bq-(\bq-\bfe_{\bomega})^+$,
$\bz\eqdef \bq-(\bq-\bfe_{\bomega})^++\bfe_{\bmu}$,
$\bw\eqdef \bq-(\bq-\bfe_{\bnu})^+$.
We also have $\bx,\by,\bz,\bw\in\{0,1\}^{4}$, $\bx\leq \bfe_{\bmu}+\bfe_{\bnu}$, $\by\leq\bfe_{\brho}+\bfe_{\bomega}$, $\bx+\by=\bz+\bw$,
and thus we obtain
\begin{align*}
V_{k+1}(\bq+\bfe_{\bmu})+V_{k+1}(\bq+\bfe_{\bmu}+\bfe_{\brho}) \geq & c_{\bmu}+c_{\bnu}\cdot \indic_{\{q_{\bnu}>0\}} + c_{\brho}+c_{\bomega}\cdot \indic_{\{q_{\bomega}>0\}} + \beta\E[V_{k}(\bQ+\bx)+V_{k}(\bQ+\by)]\\
\geq & c_{\bmu}+c_{\bnu}\cdot \indic_{\{q_{\bnu}>0\}} + c_{\brho}+c_{\bomega}\cdot \indic_{\{q_{\bomega}>0\}} + \beta\E[V_{k}(\bQ+\bz)+V_{k}(\bQ+\bw)] \\
=& V_{k+1}(\bq+2\bfe_{\bmu})+V_{k+1}(\bq+\bfe_{\brho}) ,
\end{align*}
where the second inequality follows from Lemma~\ref{LEMMA:AUX_INEQUALITY_ALT}.

Finally, we do not need to consider the fourth case. From the induction hypothesis, we know that $V_k(\cdot)$ satisfies \eqref{eq:two_queues}, \eqref{eq:three_queues_other} and \eqref{eq:three_queues_same}, and thus from Lemma \ref{lem:IH_switching_curve} we know that if the optimal action on $(\bq+2\bfe_{\bmu})$ is to serve queues $\brho$ and $\bomega$, then an optimal action on $\bq+\bfe_{\brho}$ is to again serve queues  $\brho$ and $\bomega$, and therefore the fourth case can be eliminated. 

Hence, \eqref{eq:two_queues} holds for the $(k+1)$-th value function $V_{k+1}$.

%
\textbf{Proof of \eqref{eq:three_queues_other} for $V_{k+1}$.}
Analogously following the proof of \eqref{eq:two_queues} for $V_{k+1}$,
we prove \eqref{eq:three_queues_other} of Proposition~\ref{PROP:SWITCHING_CURVE} for $n=k+1$.
The right-hand side of this equation 
involves $V_{k+1}(\bq+2\bfe_{\bmu}+\bfe_{\bnu})$ and  $V_{k+1}(\bq+\bfe_{\brho})$. By the definition of $V_{k+1}(\cdot)$ in \eqref{eq:BellmanEq}, we need to work with the maximization problem on the right-hand side of \eqref{eq:BellmanEq}. 
We consider four cases based on the optimal schedules at $\bq+2\bfe_{\bmu}$ and $\bq+\bfe_{\brho}$ in this maximization problem:
(1) Both optimal actions are to serve queues $\bmu$ and $\bnu$;
(2) Both optimal actions are to serve queues $\brho$ and $\bomega$;
(3) The optimal action on $\bq+2\bfe_{\bmu}+\bfe_{\bnu}$ is to serve queues $\bmu$ and $\bnu$, and the optimal action on $\bq+\bfe_{\brho}$ is to serve queues $\brho$ and $\bomega$;
(4) The optimal action on $\bq+2\bfe_{\bmu}+\bfe_{\bnu}$ is to serve queues $\brho$ and $\bomega$, and the optimal action on $\bq+\bfe_{\brho}$ is to serve queues $\bmu$ and $\bnu$.
Once again, by the induction hypothesis, we know from Lemma \ref{lem:IH_switching_curve} that if an optimal action on $\bq+2\bfe_{\bmu}+\bfe_{\bnu}$ is to serve queues $\brho$ and $\bomega$, it will continue to be an optimal action for $\bq+\bfe_{\brho}$, and thus we can ignore the fourth case.
We prove \eqref{eq:three_queues_other} for the $(k+1)$-th value function dealing with all three remaining cases.

First, suppose that both optimal actions are to transmit packets in queues $\bmu$ and $\bnu$ in the $(k+1)$-th value iteration. 
If $q_{\bmu}\geq 1$ and $q_{\bnu}\geq 1$, we have
\begin{align*}
V_{k+1}(\bq+2\bfe_{\bmu}+\bfe_{\bnu})+V_{k+1}(\bq+\bfe_{\brho}) &\\
& \hspace*{-1.0in} =c_{\bmu}+c_{\bnu}+\beta\E[ V_k(\bq+\bA+\bfe_{\bmu})] + c_{\bmu}+c_{\bnu}+\beta\E[ V_k( \bq+\bA-\bfe_{\bmu}-\bfe_{\bnu}+\bfe_{\brho})]\\
& \hspace*{-1.0in} \leq c_{\bmu}+c_{\bnu}+\beta\E[ V_k(\bq+\bA-\bfe_{\bnu}+\bfe_{\brho}) ] + c_{\bmu}+c_{\bnu}+\beta\E[ V_k(\bq+\bA)]\\
& \hspace*{-1.0in} \leq V_{k+1}(\bq+\bfe_{\bmu}+\bfe_{\brho})+V_{k+1}(\bq+\bfe_{\bmu}+\bfe_{\bnu}) ;
\end{align*}
here the first inequality follows from the induction hypothesis (substituting $\bq+\bA-\bfe_{\bmu}-\bfe_{\bnu}$ for $\bq$ in \eqref{eq:three_queues_other}
for $V_k$) and the second inequality follows from the definition of the value iteration.
If $q_{\bmu}\geq 1$ and $q_{\bnu}=0$, we obtain
\begin{align*}
V_{k+1}(\bq+2\bfe_{\bmu}+\bfe_{\bnu})+V_{k+1}(\bq+\bfe_{\brho}) =& c_{\bmu}+c_{\bnu}+\beta\E[ V_k(\bq+\bA+\bfe_{\bmu})]
+ c_{\bmu}+\beta\E[ V_k( \bq+\bA-\bfe_{\bmu}+\bfe_{\brho})]\\
\leq& c_{\bmu}+\beta\E[ V_k(\bq+\bA+\bfe_{\brho}) ] + c_{\bmu}+c_{\bnu}+\beta\E[ V_k(\bq+\bA)]\\
\leq& V_{k+1}(\bq+\bfe_{\bmu}+\bfe_{\brho})+V_{k+1}(\bq+\bfe_{\bmu}+\bfe_{\bnu}),
\end{align*}
where the first inequality follows from the induction hypothesis (substituting $\bq+\bA-\bfe_{\bmu}$ for $\bq$ in \eqref{eq:two_queues} for $V_k$)
and the second inequality follows from the definition of the value iteration.
Lastly, if $q_{\bmu}=0$, we have 
\begin{align*}
V_{k+1}(\bq+2\bfe_{\bmu}+\bfe_{\bnu})+V_{k+1}(\bq+\bfe_{\brho}) &\\
& \hspace*{-1.0in} =c_{\bmu}+c_{\bnu}+\beta\E[ V_k(\bq+\bA+\bfe_{\bmu})] + c_{\bnu}\cdot \indic_{\{q_{\bnu}>0\}}+\beta\E[ V_k( (\bq-\bfe_{\bnu})^++\bA+\bfe_{\brho})]\\
& \hspace*{-1.0in} \leq c_{\bmu}+c_{\bnu}\cdot \indic_{\{q_{\bnu}>0\}}+\beta\E[ V_k( (\bq-\bfe_{\bnu})^++\bA+\bfe_{\brho})] + c_{\bmu}+c_{\bnu}+\beta\E[ V_k(\bq+\bA)]\\
& \hspace*{-1.0in} \leq V_{k+1}(\bq+\bfe_{\bmu}+\bfe_{\brho})+V_{k+1}(\bq+\bfe_{\bmu}+\bfe_{\bnu}) ;
\end{align*}
here the first inequality follows from Proposition~\ref{prop:trivial_boundary} and the second inequality follows from the definition of the value iteration.

Second, assume that both optimal actions are to transmit packets in queues $\brho$ and $\bomega$ in the $(k+1)$-th value iteration.
If $q_{\brho}\geq 1$, we obtain
\begin{align*}
V_{k+1}(\bq+2\bfe_{\bmu}+\bfe_{\bnu})+V_{k+1}(\bq+\bfe_{\brho}) &\\
& \hspace*{-2.3in} =c_{\brho}+c_{\bomega}\cdot \indic_{\{q_{\bomega}>0\}} +\beta\E[ V_k((\bq-\bfe_{\bomega})^++\bA-\bfe_{\brho}+2\bfe_{\bmu}+\bfe_{\bnu})]
+c_{\brho}+c_{\bomega}\cdot \indic_{\{q_{\bomega}>0\}}+\beta\E[ V_k((\bq-\bfe_{\bomega})^++\bA)]\\
& \hspace*{-2.3in} \leq c_{\brho}+c_{\bomega}\cdot \indic_{\{q_{\bomega}>0\}}+\beta\E[ V_k((\bq-\bfe_{\bomega})^++\bA+\bfe_{\bmu})] +c_{\brho}+c_{\bomega}\cdot \indic_{\{q_{\bomega}>0\}}
+\beta\E[ V_k((\bq-\bfe_{\bomega})^++\bA-\bfe_{\brho}+\bfe_{\bmu}+\bfe_{\bnu})]\\
& \hspace*{-2.3in} \leq V_{k+1}(\bq+\bfe_{\bmu}+\bfe_{\brho})+V_{k+1}(\bq+\bfe_{\bmu}+\bfe_{\bnu}),
\end{align*}
where the first inequality follows from the induction hypothesis (substituting $(\bq-\bfe_{\bomega})^++\bA-\bfe_{\brho}$ for $\bq$ in \eqref{eq:three_queues_other} for $V_k$) and the second inequality follows from the definition of the value iteration.
On the other hand, if $q_{\brho}=0$, we have
\begin{align*}
V_{k+1}(\bq+2\bfe_{\bmu}+\bfe_{\bnu})+V_{k+1}(\bq+\bfe_{\brho}) &\\
& \hspace*{-2.2in} =c_{\bomega}\cdot \indic_{\{q_{\bomega}>0\}}+\beta\E[ V_k((\bq-\bfe_{\bomega})^++\bA+2\bfe_{\bmu}+\bfe_{\bnu})]
+c_{\brho}+c_{\bomega}\cdot \indic_{\{q_{\bomega}>0\}}+\beta\E[ V_k((\bq-\bfe_{\bomega})^++\bA)]\\
& \hspace*{-2.2in} \leq c_{\brho}+c_{\bomega}\cdot \indic_{\{q_{\bomega}>0\}}+\beta\E[ V_k((\bq-\bfe_{\bomega})^++\bA+\bfe_{\bmu})]
+c_{\bomega}\cdot \indic_{\{q_{\bomega}>0\}}+\beta\E[ V_k((\bq-\bfe_{\bomega})^++\bA+\bfe_{\bmu}+\bfe_{\bnu})]\\
& \hspace*{-2.2in} \leq V_{k+1}(\bq+\bfe_{\bmu}+\bfe_{\brho})+V_{k+1}(\bq+\bfe_{\bmu}+\bfe_{\bnu}) ;
\end{align*}
here the first inequality follows from the induction hypothesis (substituting $(\bq-\bfe_{\bomega})^++\bA$ for $\bq$ in \eqref{eq:for_three_different} for $V_k$)
and the second inequality follows from the definition of the value iteration.

Finally, suppose that the optimal action on $\bq+2\bfe_{\bmu}+\bfe_{\bnu}$ is to serve packets in queues $\bmu$ and $\bnu$
and the optimal action on $\bq+\bfe_{\brho}$ is to transmit packets in queues $\brho$ and $\bomega$ in $(k+1)$-th value iteration. 
Then, we obtain
\begin{align*}
V_{k+1}(\bq+2\bfe_{\bmu}+\bfe_{\bnu})+V_{k+1}(\bq+\bfe_{\brho}) =& c_{\bmu}+c_{\bnu}+\beta\E[ V_k(\bq+\bA+\bfe_{\bmu})]\\
&\quad +c_{\brho}+c_{\bomega}\cdot \indic_{\{q_{\bomega}>0\}}+\beta\E[ V_k((\bq-\bfe_{\bomega})^++\bA)]\\
=& (c_{\bmu}+c_{\bnu}+c_{\brho}+c_{\bomega}\cdot \indic_{\{q_{\bomega}>0\}}) +\beta\E[V_{k}(\bQ+\bz)+V_{k}(\bQ+\bw)],\\
V_{k+1}(\bq+\bfe_{\bmu}+\bfe_{\brho})+V_{k+1}(\bq+\bfe_{\bmu}+\bfe_{\bnu}) \geq& c_{\brho}+c_{\bomega}\cdot \indic_{\{q_{\bomega}>0\}}+\beta\E[ V_k((\bq-\bfe_{\bomega})^++\bA+\bfe_{\bmu})]\\
&\quad +c_{\bmu}+c_{\bnu}+\beta\E[ V_k(\bq+\bA)]\\
=&(c_{\bmu}+c_{\bnu}+c_{\brho}+c_{\bomega}\cdot \indic_{\{q_{\bomega}>0\}})+\beta\E[V_{k}(\bQ+\bx)+V_{k}(\bQ+\by)],
\end{align*}
where
$\bQ\eqdef (\bq-\bfe_{\bomega})^++\bA$,
$\bx\eqdef \bfe_{\bmu}$,
$\by\eqdef \bq-(\bq-\bfe_{\bomega})^+$,
$\bz\eqdef \bfe_{\bmu}+\bq-(\bq-\bfe_{\bomega})^+$,
$\bw\eqdef \bzero$.
We also have $\bx,\by,\bz,\bw\in\{0,1\}^{4}$, $\bx\leq \bfe_{\bmu}+\bfe_{\bnu}$, $\by\leq\bfe_{\brho}+\bfe_{\bomega}$, $\bx+\by=\bz+\bw$,
and thus we obtain
\begin{equation*}
V_{k+1}(\bq+2\bfe_{\bmu}+\bfe_{\bnu})+V_{k+1}(\bq+\bfe_{\brho}) \leq  V_{k+1}(\bq+\bfe_{\bmu}+\bfe_{\brho})+V_{k+1}(\bq+\bfe_{\bmu}+\bfe_{\bnu}) ,
\end{equation*}
which follows from Lemma~\ref{LEMMA:AUX_INEQUALITY_ALT}.

Hence, \eqref{eq:three_queues_other} holds for the $(k+1)$-th value function $V_{k+1}$.

%
\textbf{Proof of \eqref{eq:three_queues_same} for $V_{k+1}$.}
Analogously following the proof of \eqref{eq:three_queues_other} for $V_{k+1}$,
we prove \eqref{eq:three_queues_same} of Proposition~\ref{PROP:SWITCHING_CURVE} for $n=k+1$.
The right-hand side of this equation 
involves $V_{k+1}(\cdot)$ at $\bq+2\bfe_{\bmu}$ and  $\bq+\bfe_{\brho}+\bfe_{\bomega}$. We consider the four cases corresponding to the maximizers at each of these values in the $(k+1)$-th iteration of value iteration \eqref{eq:BellmanEq}. When the optimal action on $\bq+2\bfe_{\bmu}$ is to serve queues  $\brho$ and $\bomega$, we know from the induction hypothesis and Lemma \ref{lem:IH_switching_curve} that serving  queues  $\brho$ and $\bomega$ continues to be an optimal action on $\bq+\bfe_{\brho}+\bfe_{\bomega}$. Therefore, it is sufficient to consider the following three cases based on the optimal schedules at $\bq+2\bfe_{\bmu}$ and $\bq+\bfe_{\brho}$:
(1) Both optimal actions are to serve queues $\bmu$ and $\bnu$;
(2) Both optimal actions are to serve queues $\brho$ and $\bomega$;
(3) The optimal action on $\bq+2\bfe_{\bmu}$ is to serve queues $\bmu$ and $\bnu$ and the optimal action on $\bq+\bfe_{\brho}+\bfe_{\bomega}$ is to serve queues $\brho$ and $\bomega$.
We prove \eqref{eq:three_queues_same} for the $(k+1)$-th value function dealing with all three cases.

First, suppose that both optimal actions are to transmit packets in queues $\bmu$ and $\bnu$ in the $(k+1)$-th value iteration. If $q_{\bmu}\geq 1$, we have
\begin{align*}
V_{k+1}(\bq+2\bfe_{\bmu})+V_{k+1}(\bq+\bfe_{\brho}+\bfe_{\bomega}) &\\
& \hspace*{-2.5in} =c_{\bmu}+c_{\bnu}\cdot \indic_{\{q_{\bnu}>0\}}+\beta\E[ V_k((\bq-\bfe_{\bnu})^++\bA+\bfe_{\bmu})] + c_{\bmu}+c_{\bnu}\cdot \indic_{\{q_{\bnu}>0\}}
+\beta\E[ V_k((\bq-\bfe_{\bnu})^++\bA-\bfe_{\bmu}+\bfe_{\brho}+\bfe_{\bomega})]\\
& \hspace*{-2.5in} \leq c_{\bmu}+c_{\bnu}\cdot \indic_{\{q_{\bnu}>0\}}+\beta\E[ V_k((\bq-\bfe_{\bnu})^++\bA+\bfe_{\brho}+\bfe_{\bomega}) ]
+ c_{\bmu}+c_{\bnu}\cdot \indic_{\{q_{\bnu}>0\}}+\beta\E[ V_k((\bq-\bfe_{\bnu})^++\bA)]\\
& \hspace*{-2.5in} \leq V_{k+1}(\bq+\bfe_{\bmu}+\bfe_{\brho}+\bfe_{\bomega})+V_{k+1}(\bq+\bfe_{\bmu}) ;
\end{align*}
here the first inequality follows from the induction hypothesis (substituting $(\bq-\bfe_{\bnu})^++\bA-\bfe_{\bmu}$ for $\bq$ in \eqref{eq:three_queues_same} for $V_k$)
and the second inequality follows from the definition of the value iteration.
On the other hand, if $q_{\bmu}=0$, we obtain
\begin{align*}
V_{k+1}(\bq+2\bfe_{\bmu})+V_{k+1}(\bq+\bfe_{\brho}+\bfe_{\bomega}) &\\
& \hspace*{-2.0in} =c_{\bmu}+c_{\bnu}\cdot \indic_{\{q_{\bnu}>0\}}+\beta\E[ V_k((\bq-\bfe_{\bnu})^++\bA+\bfe_{\bmu})]
+ c_{\bnu}\cdot \indic_{\{q_{\bnu}>0\}}+\beta\E[ V_k((\bq-\bfe_{\bnu})^++\bA+\bfe_{\brho}+\bfe_{\bomega})]\\
& \hspace*{-2.0in} \leq c_{\bmu}+c_{\bnu}\cdot \indic_{\{q_{\bnu}>0\}}+\beta\E[ V_k((\bq-\bfe_{\bnu})^++\bA+\bfe_{\brho}+\bfe_{\bomega})]
+ c_{\bmu}+c_{\bnu}\cdot \indic_{\{q_{\bnu}>0\}}+\beta\E[ V_k((\bq-\bfe_{\bnu})^++\bA)]\\
& \hspace*{-2.0in} \leq V_{k+1}(\bq+\bfe_{\bmu}+\bfe_{\brho}+\bfe_{\bomega})+V_{k+1}(\bq+\bfe_{\bmu}),
\end{align*}
where the first inequality follows from Proposition~\ref{prop:trivial_boundary} and the second inequality follows from the definition of the value iteration.

Second, assume that both optimal actions are to transmit packets in queues $\brho$ and $\bomega$ in the $(k+1)$-th value iteration.
However, if $q_{\brho}=0$ and $q_{\bomega}=0$, the optimal action cannot be optimal on $\bq+2\bfe_{\bmu}$, as this would imply transmitting nothing to be optimal which is obviously not optimal. 
Hence,
one of $q_{\brho}$ and $q_{\bomega}$ should not be zero.
If $q_{\brho}\geq 1$ and $q_{\bomega}\geq 1$, we have
\begin{align*}
V_{k+1}(\bq+2\bfe_{\bmu})+V_{k+1}(\bq+\bfe_{\brho}+\bfe_{\bomega}) &\\
& \hspace*{-1.0in} =c_{\brho}+c_{\bomega}+\beta\E[ V_k(\bq-\bfe_{\brho}-\bfe_{\bomega}+\bA+2\bfe_{\bmu})] + c_{\brho}+c_{\bomega}+\beta\E[ V_k(\bq+\bA)]\\
& \hspace*{-1.0in} \leq c_{\brho}+c_{\bomega}+\beta\E[ V_k(\bq+\bA+\bfe_{\bmu}) ] + c_{\brho}+c_{\bomega}+\beta\E[ V_k(\bq-\bfe_{\brho}-\bfe_{\bomega}+\bA+\bfe_{\bmu})]\\
& \hspace*{-1.0in} \leq V_{k+1}(\bq+\bfe_{\bmu}+\bfe_{\brho}+\bfe_{\bomega})+V_{k+1}(\bq+\bfe_{\bmu}) ;
\end{align*}
here the first inequality follows from the induction hypothesis (substituting $\bq-\bfe_{\brho}-\bfe_{\bomega}+\bA$ for $\bq$ in \eqref{eq:three_queues_same} for $V_k$)
and the second inequality follows from the definition of the value iteration. 
If only one of $q_{\brho}$ and $q_{\bomega}$ is zero, assume without loss of generality that $q_{\brho}\geq 1$ and $q_{\bomega}=0$.
We then obtain
 \begin{align*}
V_{k+1}(\bq+2\bfe_{\bmu})+V_{k+1}(\bq+\bfe_{\brho}+\bfe_{\bomega}) =& c_{\brho}+\beta\E[ V_k(\bq-\bfe_{\brho}+\bA+2\bfe_{\bmu})]
+ c_{\brho}+c_{\bomega}+\beta\E[ V_k(\bq+\bA)]\\
\leq& c_{\brho}+c_{\bomega}+\beta\E[ V_k(\bq+\bA+\bfe_{\bmu}) ]
+ c_{\brho}+\beta\E[ V_k(\bq-\bfe_{\brho}+\bA+\bfe_{\bmu})]\\
\leq& V_{k+1}(\bq+\bfe_{\bmu}+\bfe_{\brho}+\bfe_{\bomega})+V_{k+1}(\bq+\bfe_{\bmu}),
\end{align*}
where the first inequality follows from the induction hypothesis (substituting $\bq-\bfe_{\brho}+\bA$ for $\bq$ in \eqref{eq:two_queues} for $V_k$)
and the second inequality follows from the definition of the value iteration. 

Finally, suppose that the optimal action on $(\bq+2\bfe_{\bmu})$ is to serve packets in queues $\bmu$ and $\bnu$
and the optimal action on $\bq+\bfe_{\brho}+\bfe_{\bomega}$ is to transmit packets in queues $\brho$ and $\bomega$ in the $(k+1)$-th value iteration.
We then have
\begin{align*}
V_{k+1}(\bq+2\bfe_{\bmu})+V_{k+1}(\bq+\bfe_{\brho}+\bfe_{\bomega}) \qquad\quad &\\
& \hspace*{-2.6in} = c_{\bmu}+c_{\bnu}\cdot \indic_{\{q_{\bnu}>0\}}+\beta\E[ V_k((\bq-\bfe_{\bnu})^++\bA+\bfe_{\bmu})] + c_{\brho}+c_{\bomega}+\beta\E[ V_k(\bq+\bA)]\\
& \hspace*{-2.6in} = (c_{\bmu}+c_{\bnu}\cdot \indic_{\{q_{\bnu}>0\}}+c_{\brho}+c_{\bomega})  
+ \beta\E[ V_k(\bQ+\bz)+V_k(\bQ+\bw)],\\
V_{k+1}(\bq+\bfe_{\bmu})+V_{k+1}(\bq+\bfe_{\bmu}+\bfe_{\brho}+\bfe_{\bomega}) \qquad &\\
& \hspace*{-2.6in} \geq c_{\bmu}+c_{\bnu}\cdot \indic_{\{q_{\bnu}>0\}}+\beta\E[ V_k((\bq-\bfe_{\bnu})^++\bA)] + c_{\brho}+c_{\bomega}+\beta\E[ V_k(\bq+\bA+\bfe_{\bmu}) ]\\
& \hspace*{-2.6in} = (c_{\bmu}+c_{\bnu}\cdot \indic_{\{q_{\bnu}>0\}}+c_{\brho}+c_{\bomega}) 
+ \beta\E[ V_k(\bQ+\bx)+V_k(\bQ+\by)],
\end{align*}
where
$\bQ\eqdef (\bq-\bfe_{\bnu})^++\bA$,
$\bx\eqdef \bfe_{\bmu}+\bq-(\bq-\bfe_{\bnu})^+$,
$\by\eqdef \bzero$,
$\bz\eqdef \bfe_{\bmu}$,
$\bw\eqdef \bq-(\bq-\bfe_{\bnu})^+$.
We also have $\bx,\by,\bz,\bw\in\{0,1\}^{4}$, $\bx\leq \bfe_{\bmu}+\bfe_{\bnu}$, $\by\leq\bfe_{\brho}+\bfe_{\bomega}$, $\bx+\by=\bz+\bw$,
and thus we obtain
\begin{align*}
V_{k+1}(\bq+\bfe_{\bmu})+V_{k+1}(\bq+\bfe_{\bmu}+\bfe_{\brho}+\bfe_{\bomega})
	\geq&  V_{k+1}(\bq+2\bfe_{\bmu})+V_{k+1}(\bq+\bfe_{\brho}+\bfe_{\bomega}),
\end{align*}
which follows from Lemma~\ref{LEMMA:AUX_INEQUALITY_ALT}.

Hence, \eqref{eq:three_queues_same} holds for the $(k+1)$-th value function $V_{k+1}$.

\subsection{Identifying the Optimal Policy}

\subsubsection{Proof of Proposition~\ref{PROP:SYMMETRIC}.}
The proof is by induction on $n$.
First, $V_0$ satisfies \eqref{eq:symmetric} because $V_0=0$. Next, suppose that \eqref{eq:symmetric} holds for $n=k$.
Let $\bq=x\,\bfe_{\brho}+y\,\bfe_{\bomega}+z\,\bfe_{\bmu}+w\,\bfe_{\bnu}$ and $\bq^\prime=z\,\bfe_{\brho}+w\,\bfe_{\bomega}+x\,\bfe_{\bmu}+y\,\bfe_{\bnu}$.
We then have
\begin{align*}
        r(\bq,\bs_1)=r(\bq^\prime,\bs_2), \qquad & \qquad
        \E[V_k((\bq-\bs_1)^++\bA)]=\E[V_k((\bq^\prime-\bs_2)^++\bA)],\\
        r(\bq,\bs_2)=r(\bq^\prime,\bs_1), \qquad & \qquad
        \E[V_k((\bq-\bs_2)^++\bA)]=\E[V_k((\bq^\prime-\bs_1)^++\bA)] .
\end{align*}
From the i.i.d.\ assumption and the induction hypothesis on $n=k$, the Bellman equation~\eqref{eq:BellmanEq} yields $V_{k+1}(\bq)=V_{k+1}(\bq^\prime)$, which means that \eqref{eq:symmetric} holds for any value function.

\subsubsection{Proof of Theorem~\ref{thm:optimal}.}
Using the notation that $\IndSet=\{\brho,\bomega,\bmu,\bnu\}$ with $\brho\#\bmu$ and $\brho\#\bnu$,
we consider each of the three possible regional cases as follows.

\begin{enumerate}
\item[\textit{Case (i)}:] 
Since the possible maximum reward $r_{\max}$ is $2$, state $\bq$ is in the interior if and only if $r(\bq,\bs)=2$ for some schedule $\bs\in\cP$.
Therefore, if a size-$2$ schedule exists, the state is in the interior and by Theorem~\ref{thm:interior}, selecting such a schedule is optimal.
\item[\textit{Case (ii)}:]
Without loss of generality, we assume that $\bq=x\,\bfe_{\brho}+z\,\bfe_{\bmu}$ with $x,z\in\Nats$ and $x\geq z$.
Further, let $\bs_1=\bfe_{\brho}+\bfe_{\bomega}$ and $\bs_2=\bfe_{\bmu}+\bfe_{\bnu}$.
Then, assuming first that $x=z$, we have
\begin{align*}
r(\bq,\bs_1)+\beta\E[V_n((\bq-\bs_1)^+ +\bA)] &=1+\beta\E[V_n((x-1+A_{\brho})\bfe_{\brho}+A_{\bomega}\bfe_{\bomega}+(x+A_{\bmu})\bfe_{\bmu}+A_{\bnu}\bfe_{\bnu} )]\\
&=1+\beta\E[V_n((x+A_{\bmu})\bfe_{\brho}+A_{\bnu}\bfe_{\bomega}+(x-1+A_{\brho})\bfe_{\bmu}+A_{\bomega}\bfe_{\bnu} )] \\
&=1+\beta\E[V_n((x+A_{\brho})\bfe_{\brho}+A_{\bomega}\bfe_{\bomega}+(x-1+A_{\bmu})\bfe_{\bmu}+A_{\bnu}\bfe_{\bnu} )] \\
& =r(\bq,\bs_2)+\beta\E[V_n((\bq-\bs_2)^++\bA)],
\end{align*}
	where the second equation comes from \eqref{eq:symmetric} and the third equation follows from the fact that $A_{\bmu}$, $A_{\bnu}$ $A_{\brho}$ and $A_{\omega}$ are
i.i.d.
Any schedule is therefore optimal.
Second, when $x>z$, we apply Theorem~\ref{thm:switching_curve} to $\bq=z\,\bfe_{\rho}+z\,\bfe_{\bmu}$, from which serving queue $\brho$ is optimal.
Then, the schedule is also optimal for $\bq+(x-z)\,\bfe_{\brho}=x\,\bfe_{\brho}+z\,\bfe_{\bmu}$.
\item[\textit{Case (iii)}:] 
If a state does not belong to any of above cases, it is in the trivial boundary. Therefore, by Theorem~\ref{thm:trivial_boundary}, the unique size-$1$ schedule is optimal.
%
%
\end{enumerate}

\subsubsection{Proof of Theorem~\ref{thm:convergence}.}
From the proof of Proposition~\ref{prop:alt_formulation}, we obtain
\begin{equation*}
(1-\beta)J_\beta(\bq,\pi_\ell)=c^\pi(0)+g-\beta \tilde{J}_\beta(\bq, \pi_\ell), \qquad\qquad (1-\beta)V^*_\beta=c^\pi(0)+g-\beta \tilde{V}^*_\beta,
\end{equation*}
where $g=\sum_{t=0}^\infty\beta^{t+1}\,\E[ \sum_{\brho\in\IndSet}c_{\brho} A_{\brho}(t)]$.
Subtracting the second equation from the first yields
\begin{equation*}
	\norm{J_\beta(\,\cdot\,,\pi_\ell)-V_\beta^*}=\frac{\beta}{1-\beta}\norm{\tilde{J}_\beta(\,\cdot\,,\pi_\ell)-\tilde{V}^*_\beta}.
\end{equation*}
On the other hand, from Theorem 6.3.1 in \citep{Pute05}, we have that $\tilde{J}(\,\cdot\,,\pi_\ell)$ converges to $\tilde{V}^*_\beta$ and
$\norm{\tilde{J}(\,\cdot\,,\pi_\ell)-\tilde{V}_\beta^*}<\frac{1-\beta}{\beta}\varepsilon$ when \eqref{eq:conv_assume} holds, which implies the desired results.

\section{Computational Experiments}
\label{sec:experiments}
Our main theoretical results establish an optimal scheduling policy for the $2\times 2$ input-queued switch model.
To further investigate issues of delay-cost optimality, we consider in this section a representative sample of results from numerous computational experiments
on the performance of these optimal solutions to the optimization problem \eqref{prob:opt_tilde_J} in comparison with variants of the MaxWeight scheduling
policy.
The case of symmetric arrivals and unit costs across all queues is studied first,
in which case we have an explicit optimal policy from the results of Section~\ref{sec:optimal:symmetric},
followed by a study of the general case for arrival and cost vectors across all queues,
in which case we have an asymptotically optimal policy from the results of Section~\ref{sec:optimal:general}.

\subsection{Symmetric Arrivals and Unit Costs}
\label{sec:experiments:symmetric}
For the case of symmetric arrivals and unit costs, our theoretical results show that Algorithm~1 provides an explicit optimal solution to
both problems~\eqref{prob:opt_tilde_J} and \eqref{prob:opt_J} for any discount factor $\beta \in (0,1)$.
It is important to note key differences in decisions between the policy of Algorithm~1 and the MaxWeight scheduling policy.
Both policies will take similar actions in the interior region (i.e., when all four queues are non-empty) given that Algorithm~1 will choose either of the two
schedules and the MaxWeight policy will choose the one with the highest queue-length weight (i.e., the summation of both queue lengths), consistent with case (i) in Algorithm~1.
On the other hand, outside of the interior region, there can exist situations consisting of one size-$2$ schedule having queue-length weight $w_2$ and one size-$1$ schedule
having queue-length weight $w_1$ where $w_1 > w_2$ (i.e., the one queue length $w_1$ is greater than the summation $w_2$ of both queue lengths);
in such situations, Algorithm~1 will choose the size-$2$ schedule and the MaxWeight policy will choose the size-$1$ schedule, thus violating case (i) of Algorithm~1.

Now consider a policy that follows Algorithm~1 by always choosing the maximal weight schedule when in the interior region (i.e., selecting a size-$2$ schedule
with the largest summation of both queue lengths).
We call such a policy ``Maximum Size with Maximal Weight'' (MSMW) because it always selects a maximum-size schedule according to Algorithm~1 but breaks ties among
size-$2$ schedules by giving priority to a maximal weight schedule.
Next, suppose instead that the ties in case (i) of Algorithm~1 are broken using the logarithm of queue lengths as weights;
we then obtain the MSMW-log policy proposed by \cite{ShaWis_11}.
Hence, our optimal scheduling solution in the case of symmetric arrivals and unit costs (Algorithm~1) subsumes both the MSMW and MSMW-log policies,
as well as extending the optimality of the SOP policy in \citep{Saswati_Switch_buffer} beyond finite buffers.
For comparison, we also consider an additional policy of interest, denoted as the MaxSize algorithm, that consists of selecting in every time slot a schedule with a
maximal size (i.e., the maximal number of non-empty queues), breaking ties uniformly at random;
a key difference between MaxSize and Algorithm~1 is in the tie-breaking rule for case (ii), where the optimal policy serves the size-$1$ schedule with the longest queue.

To quantify the performance benefits of Algorithm~1 and to investigate issues related to the delay optimality of MaxWeight scheduling within the setting of our model
and formulation of Section~\ref{sec:model}, we use simulation to compare the performance of our optimal policy of Algorithm~1 with that of MaxWeight scheduling and
the other scheduling policy alternatives above.
Specifically, we consider a $2 \times 2$ switch with i.i.d.\ Bernoulli arrivals at each input port and unit costs ($c_{\brho}=1$ for all $\brho\in\IndSet$).
When the arrival rate for each queue is $\lambda$, the traffic intensity is $\rho = 2\lambda$. 
For the purpose of simulations, we assume that service takes place after arrivals within a time slot.
Since the research literature has focused on studying the performance of MaxWeight scheduling in steady state (as opposed to our optimal scheduling results
based on discounted delay cost) and the heavy-traffic regime, the simulation results presented in this section are based on queue lengths in both steady state and heavy traffic, in
order to provide direct comparisons with these previous MaxWeight scheduling results.
We also refer to Remark~\ref{rem:AC-MDP} concerning connections between Algorithm~1 and average-cost optimal policies.

Figure~\ref{fig1:steady-state} compares the expected summation of all queue lengths in steady state under the MSMW (Algorithm~1), MaxWeight and MaxSize algorithms.
Recalling that different optimal algorithms can be defined based on the tie-breaking rule in case (i) of Algorithm~1, we chose to consider the MSMW algorithm that
breaks ties by selecting the schedule with the maximal weight defined as the summation of the queue lengths in that schedule.
In contrast, under the MaxWeight algorithm, a schedule with a maximal weight is selected in every time slot, where the weight is calculated according to the queue
lengths;
and under the MaxSize algorithm, a schedule with a maximal size is selected in every time slot.
These key differences among the definitions of the various scheduling policies are fundamental to the performance differences exhibited among the scheduling policies.
We may also consider MSMW-log where the weights are logarithms of queue lengths.
However, in a $2 \times 2$ switch, the performance difference may not be important because of the tie-breaking rule among size-two matchings based on the proofs of
our theoretical results for the infinite-horizon discounted cost problem;
and, among size-one matchings, both MSMW and MSMW-log select the same schedule.
Figure~\ref{fig1:steady-state} also plots the universal lower bound established by~\cite{LuMaSq+18} which reduces to $\rho^2/2(1-\rho)$ for the model under
consideration.

\begin{figure}[htbp]
\begin{subfigure}{0.5\textwidth}
\centerline{\includegraphics[width=3in]{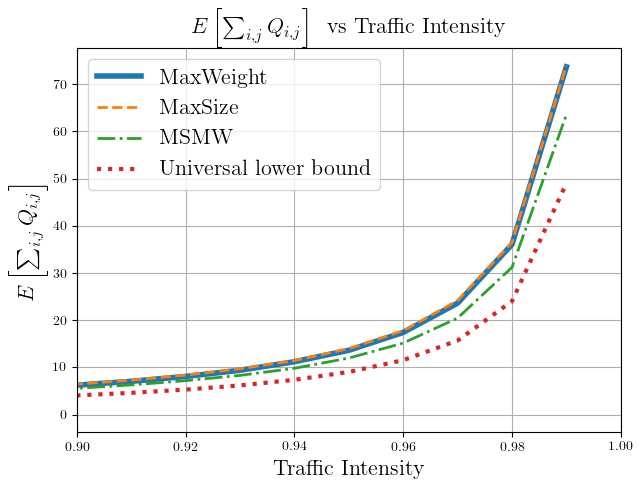}}
\caption{Steady-state comparison.}
\label{fig1:steady-state}
\end{subfigure}
\begin{subfigure}{0.5\textwidth}
\centerline{\includegraphics[width=3in]{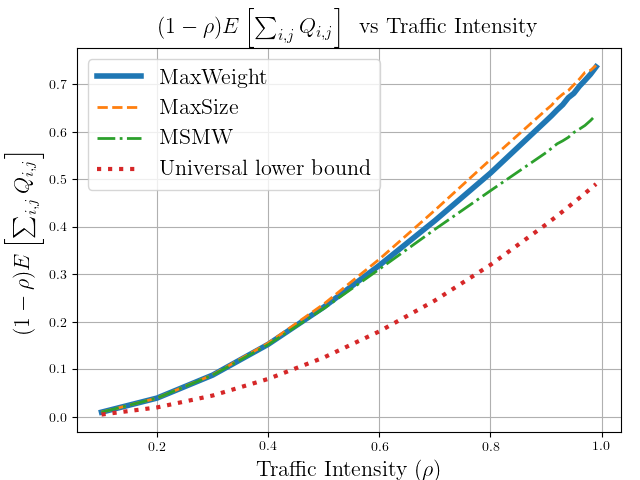}}
\caption{Heavy-traffic comparison.}
\label{fig1:heavy-traffic}
\end{subfigure}
\caption{Performance comparisons of MSMW, MaxWeight, MaxSize and universal lower bound.}
\label{fig1:symmetric}
\end{figure}

We observe that the queue lengths under all algorithms and the lower bound blow up to infinity as the traffic intensity approaches $1$,
which is as expected since the universal lower bound is $\Theta\left(\frac{1}{1-\rho}\right)$.
While it has been shown that MaxSize may not be stable for all arrival rates within the capacity region of a $2 \times 2$ switch (see~\cite{mckeown_maxsizeunstable}),
our simulation results suggest that it is stable in the case of symmetric arrivals and unit costs.
Although the performance differences among the algorithms are small in light traffic, Figure~\ref{fig1:steady-state} indicates that MSMW performs better in heavy traffic.
We also note that there is a gap between the MSMW performance and the universal lower bound.
Since MSMW, as an instance of Algorithm~1, indeed minimizes the infinite-horizon discounted queue-length problem in the case of symmetric arrivals and unit costs,
this may suggest that the universal lower bound is loose.
Such questions around the tightness of the universal lower bound have been raised in the research literature within the context of MaxWeight scheduling;
see, e.g.,~\cite{MagSri_SSY16_Switch,LuMaSq+19}.
Our results suggest that the universal lower bound is loose with respect to the optimal scheduling policy of Algorithm~1.

Given that the queue lengths are $O\left(\frac{1}{1-\rho}\right)$, we plot in Figure~\ref{fig1:heavy-traffic} the corresponding normalized queue lengths,
i.e., the queue lengths multiplied by $(1-\rho)$ which is called heavy-traffic scaling.
We use $\epsilon$ to denote $(1-\rho)=(1-2\lambda)$ where $\epsilon$ is called the heavy-traffic parameter.
The differences among the curves are more clearly evident in this figure, with MSMW performing better in heavy traffic and with not much difference between MaxWeight and MaxSize.
The limiting point of the curves in Figure~\ref{fig1:heavy-traffic} is called the heavy traffic limit,
where this limit for MaxWeight and the universal lower bound has been shown to be $0.75$ and $0.5$, respectively (see~\cite{MagSri_SSY16_Switch});
both of these results match our simulation results.
The limit under MSMW appears to be around $0.65$.


\subsection{General Case}
\label{sec:experiments:general}
For the case of general arrivals and costs, given the previously noted difficulty of deriving an explicit switching curve for the optimal scheduling policy in general,
our theoretical results show that the look-ahead policy of Section~\ref{sec:optimal:general} provides an asymptotically optimal solution to both problems~\eqref{prob:opt_tilde_J}
and \eqref{prob:opt_J} for any discount factor $\beta \in (0,1)$.
We therefore now consider the performance of the $\ell$-th look-ahead policy $\pi_\ell$ in \eqref{look-ahead.policy} of Section~\ref{sec:optimal:general}.
From value iteration on the optimization problem \eqref{prob:opt_tilde_J} starting with $V_0=0$, as in Section~\ref{sec:alt_formulation},
the class of look-ahead policies $\pi_\ell$ exploits the $\ell$-th value function as the approximation of an optimal solution.
Some of the important benefits of this class of policies include those noted in Section~\ref{sec:optimal:general}, such as only needing to determine the optimal actions
of the look-ahead policy for states in the critical boundary by exploiting our optimal results for the interior and trivial boundary, and especially
Theorem~\ref{thm:convergence} which establishes $\pi_\ell$ to be asymptotically optimal with respect to the degree of look ahead $\ell$.

We use simulation to compare the performance characteristics of our class of look-ahead policies, for different values of $\ell$, against the
corresponding performance characteristics of variants of MaxWeight scheduling.
This includes the standard MaxWeight policy, denoted by MWS, which chooses a schedule that has larger total number of packets than the other schedule.
The weighted MaxWeight policy, considered by~\cite{LuMaSq+18,LuMaSq+19} and denoted here as C-MWS, chooses a schedule that has the larger weight than the other
schedule where the weight is a linear function of the queue lengths and the cost coefficients;
e.g., packets from queues $(1,1)$ and $(2,2)$ are transmitted when $c_{11}q_{11}+c_{22}q_{11}>c_{12}q_{12}+c_{21}q_{21}$.

To quantify the performance benefits of the look-ahead policy $\pi_\ell$ as a function of $\ell$ and to investigate issues related to the delay-cost optimality of
the various scheduling policies, we obtain from simulation the expected total discounted queue length of the $\ell$-th look-ahead policy (with look ahead step size $\ell$)
and compare these performance results with the corresponding results for MWS and C-MWS.
Figure~\ref{fig2:discounted} presents a representative example of these simulation experiment results, together with 95\% confidence intervals,
under arrival rates $\lambda_{11}=0.7$, $\lambda_{22}=0.5$, $\lambda_{12}=0.2$ and $\lambda_{21}=0.29$, cost vectors $c_{11}=2$, $c_{22}=2$, $c_{12}=10$ and $c_{21}=10$,
and discount factor $\beta=0.99$, taken over $1000$ samples.

\begin{figure}[htbp]
\begin{subfigure}{0.5\textwidth}
\centerline{\includegraphics[width=3in]{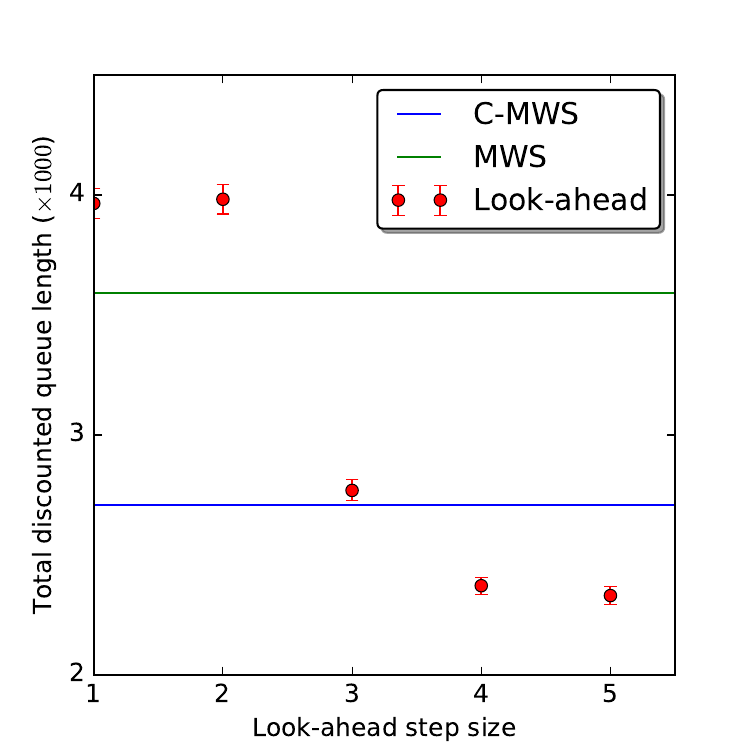}}
\caption{Total discounted queue length.}
\label{fig2:discounted}
\end{subfigure}
\begin{subfigure}{0.5\textwidth}
\centerline{
\begin{tabular}{cc}
  \hline
  Step Size & Relative optimality gap \\
  \hline
  1 & -31.76\\
  2 & -32.06\\
  3 & -2.30\\
  4 & 14.06\\
  5 & 16.07\\
  \hline
\end{tabular}
}
\caption{Relative optimality gap.}
\label{tbl:step_gap}
\end{subfigure}
\caption{Performance comparisons of $\ell$-th look-ahead policy, MaxWeight (MWS), and weighted MaxWeight (C-MWS).}
\label{fig2:general}
\end{figure}

We observe from these and related simulation experiments~~--~~taken over various arrival distributions, various arrival rates, and various cost coefficients~~--~~that
the performance of the look-ahead policy is close to the optimal performance when the step size $\ell$ is greater than or equal to $4$.
(Note that results for $\ell=6,\ldots,10$ are essentially identical to those depicted for $\ell=5$.)
We further observe from these and related simulation experiments that the look-ahead policies are good approximations to the optimal solution of problem \eqref{prob:opt_J}
even when the look ahead step size is relatively small, where the optimality gap varies from $7\%$ to $16\%$ depending on the experimental settings.
Table~\ref{tbl:step_gap} presents the relative optimality gaps between C-MWS and the look-ahead policies corresponding to the simulation results in Figure~\ref{fig2:discounted}.

\section{Conclusions}
\label{sec:conc}
Our primary goal in this paper has been to study the canonical $2\times 2$ input-queued switch and derive an optimal scheduling policy with respect to discounted
delay (equivalently, via Little's Law, queue-length) cost over an infinite time horizon.
This includes establishing that $c\mu$ is optimal in the interior region and in some cases of the critical boundary region, and that a switching curve is otherwise
optimal in the critical boundary region, for which we also provide an asymptotically optimal policy.
Our collection of theoretical results, which further include establishing theoretical properties corresponding to our optimal scheduling solution,
are expected to be of interest more broadly than input-queued switches.
We also conducted computational experiments that demonstrate and quantify the benefits of our optimal scheduling policy over alternative policies such as variants
of MaxWeight scheduling.
The fundamental insights gained from these results motivate our analysis and results for the general input-queued switch in the sequel.

%
%
%

\begin{APPENDICES}


\mss{
\section{Proof of {\normalsize \eqref{eq:for_two}} for \boldmath{\normalsize $V_{k+1}$}}
\label{sub:proof_of_eq:for_two_for_}
}
We prove that equation \eqref{eq:for_two} of Proposition~\ref{Prop:7} for $n=k+1$ holds in each case depending on the three optimal actions of the $(k+1)$-th value iteration on $\bq$ and $\bq+2\bfe_{\bmu}$.
The fourth case can again be ignored because, from the induction hypothesis and Lemma \ref{lem:IH_switching_curve}, we know that if serving queues $\bmu$ and $\bnu$ is an optimal action on $\bq$ in the $(k+1)$-th value iteration, then it will continue to be an optimal action on $\bq+2\bfe_{\bmu}$.

First, assume that both optimal actions are serving queues $\bmu$ and $\bnu$.
If $q_{\bmu}>0$, we have
\begin{align*}
V_{k+1}(\bq)+V_{k+1}(\bq+2\bfe_{\bmu}) \quad &\\
& \hspace*{-1.6in} =c_{\bmu}+c_{\bnu}\cdot \indic_{\{q_{\bnu}>0\}} + \E[V_{k}((\bq-\bfe_{\bnu})^++\bA-\bfe_{\bmu} )] 
+ c_{\bmu}+c_{\bnu}\cdot \indic_{\{q_{\bnu}>0\}} + \E[V_{k}( (\bq-\bfe_{\bnu})^++\bA+\bfe_{\bmu} )]\\
& \hspace*{-1.6in} \leq c_{\bmu}+c_{\bnu}\cdot \indic_{\{q_{\bnu}>0\}} + \E[V_{k}( (\bq-\bfe_{\bnu})^++\bA )]
+ c_{\bmu}+c_{\bnu}\cdot \indic_{\{q_{\bnu}>0\}} + \E[V_{k}( (\bq-\bfe_{\bnu})^++\bA )]
\leq 2\,V_{k+1}(\bq+\bfe_{\bmu});
\end{align*}
here the first inequality follows from the induction hypothesis (substituting $(\bq-\bfe_{\bnu})^++\bA-\bfe_{\bmu}$ for $\bq$ in \eqref{eq:for_two} for $V_k$) and the second inequality follows from the definition of the value iteration. 
On the other hand, if $q_{\bmu}=0$, we obtain
\begin{align*}
V_{k+1}(\bq)+V_{k+1}(\bq+2\bfe_{\bmu}) \quad &\\
& \hspace*{-1.6in} =c_{\bnu}\cdot \indic_{\{q_{\bnu}>0\}} + \E[V_{k}( (\bq-\bfe_{\bnu})^++\bA )]
+ c_{\bmu}+c_{\bnu}\cdot \indic_{\{q_{\bnu}>0\}} + \E[V_{k}( (\bq-\bfe_{\bnu})^++\bA+\bfe_{\bmu} )]\\
& \hspace*{-1.6in} \leq c_{\bmu}+c_{\bnu}\cdot \indic_{\{q_{\bnu}>0\}} + \E[V_{k}( (\bq-\bfe_{\bnu})^++\bA )]
+ c_{\bmu}+c_{\bnu}\cdot \indic_{\{q_{\bnu}>0\}} + \E[V_{k}( (\bq-\bfe_{\bnu})^++\bA )]
\leq
2\,V_{k+1}(\bq+\bfe_{\bmu}),
\end{align*}
where the first inequality follows from Proposition~\ref{prop:trivial_boundary}.

Second, suppose that the optimal actions are serving queues $\brho$ and $\bomega$.
Then, we have
\begin{align*}
V_{k+1}(\bq)+V_{k+1}(\bq+2\bfe_{\bmu}) \quad &\\
& \hspace*{-1.6in} =2(c_{\brho}\cdot \indic_{\{q_{\brho}>0\}}+c_{\bomega}\cdot \indic_{\{q_{\bomega}>0\}}) + \E[V_{k}( (\bq-\bfe_{\brho}-\bfe_{\bomega})^++\bA )]
+ \E[V_{k}( (\bq-\bfe_{\brho}-\bfe_{\bomega})^++\bA+2\bfe_{\bmu} )]\\
& \hspace*{-1.6in} \leq2(c_{\brho}\cdot \indic_{\{q_{\brho}>0\}}+c_{\bomega}\cdot \indic_{\{q_{\bomega}>0\}}) + 2\E[V_{k}( (\bq-\bfe_{\brho}-\bfe_{\bomega})^++\bA+\bfe_{\bmu} )]
\leq 2\,V_{k+1}(\bq+\bfe_{\bmu});
\end{align*}
here the first inequality follows from the induction hypothesis (substituting $(\bq-\bfe_{\brho}-\bfe_{\bomega})^++\bA$ for $\bq$ in \eqref{eq:for_two} for $V_k$) and the second inequality follows from the definition of the value iteration. 

Finally, assume that the optimal action on $\bq$ is to serve queues $\brho$ and $\bomega$, and the optimal action on $\bq+2\bfe_{\bmu}$ is to serve queues $\bmu$ and $\bnu$.
We then obtain
\begin{align*}
V_{k+1}(\bq+2\bfe_{\bmu})+V_{k+1}(\bq) &\\
& \hspace*{-1.4in} =c_{\bmu}+c_{\bnu}\cdot \indic_{\{q_{\bnu}>0\}} + \beta\E[V_{k}( (\bq-\bfe_{\bnu})^++\bA+\bfe_{\bmu} )]
+ c_{\brho}\cdot \indic_{\{q_{\brho}>0\}}+c_{\bomega}\cdot \indic_{\{q_{\bomega}>0\}}
+ \beta\E[V_{k}( (\bq-\bfe_{\brho}-\bfe_{\bomega})^++\bA )]\\
& \hspace*{-1.4in} =(c_{\bmu}+c_{\bnu}\cdot \indic_{\{q_{\bnu}>0\}}+c_{\brho}\cdot \indic_{\{q_{\brho}>0\}}+c_{\bomega}\cdot \indic_{\{q_{\bomega}>0\}})
+ \E[V_{k}(\bQ+\bz)+V_{k}(\bQ+\bw)],\\
2 V_{k+1}(\bq+\bfe_{\bmu}) \qquad\qquad\quad &\\
& \hspace*{-1.4in} \geq c_{\bmu}+c_{\bnu}\cdot \indic_{\{q_{\bnu}>0\}} + \beta\E[V_{k}( (\bq-\bfe_{\bnu})^++\bA )]
+ c_{\brho}\cdot \indic_{\{q_{\brho}>0\}}+c_{\bomega}\cdot \indic_{\{q_{\bomega}>0\}}
+ \beta\E[V_{k}( (\bq-\bfe_{\brho}-\bfe_{\bomega})^++\bA+\bfe_{\bmu} )]\\
& \hspace*{-1.4in} =(c_{\bmu}+c_{\bnu}\cdot \indic_{\{q_{\bnu}>0\}}+c_{\brho}\cdot \indic_{\{q_{\brho}>0\}}+c_{\bomega}\cdot \indic_{\{q_{\bomega}>0\}})
+ \E[V_{k}(\bQ+\bx)+V_{k}(\bQ+\by)],
\end{align*}
where
$\bQ\eqdef (\bq-\bfe_{\bnu}-\bfe_{\brho}-\bfe_{\bomega})^++\bA$,
$\bx\eqdef \bfe_{\bmu}+\bq-(\bq-\bfe_{\bnu})^+$,
$\by\eqdef \bq-(\bq-\bfe_{\brho}-\bfe_{\bomega})^+$,
$\bz\eqdef \bfe_{\bmu}+\bq-(\bq-\bfe_{\brho}-\bfe_{\bomega})^+$,
$\bw\eqdef \bq-(\bq-\bfe_{\bnu})^+$.
We also have $\bx,\by,\bz,\bw\in\{0,1\}^{4}$, $\bx\leq \bfe_{\bmu}+\bfe_{\bnu}$, $\by\leq\bfe_{\brho}+\bfe_{\bomega}$, $\bx+\by=\bz+\bw$,
and thus we obtain
\begin{align*}
	2 V_{k+1}(\bq+\bfe_{\bmu}) \geq  V_{k+1}(\bq+2\bfe_{\bmu})+V_{k+1}(\bq) ,
\end{align*}
which follows from Lemma~\ref{LEMMA:AUX_INEQUALITY_ALT}.

Hence, \eqref{eq:for_two} holds for the $(k+1)$-th value function $V_{k+1}$.

\mss{
\section{Proof of {\normalsize \eqref{eq:for_three_different}} for \boldmath{\normalsize $V_{k+1}$}}
\label{app:for_three_different}
}
We prove that equation \eqref{eq:for_three_different} of Proposition~\ref{Prop:7} for $n=k+1$, where $\bmu$ and $\bnu$ can be served simultaneously,
holds in each case depending on the three optimal actions of the $(k+1)$-th value iteration on $\bq$ and $\bq+2\bfe_{\bmu}+\bfe_{\bnu}$.
The fourth case can again be ignored because, from the induction hypothesis and Lemma \ref{lem:IH_switching_curve}, we know that if serving queues $\bmu$ and $\bnu$ is an optimal action on $\bq$ in the $(k+1)$-th value iteration, then it will continue to be an optimal action on $\bq+2\bfe_{\bmu}+\bfe_{\bnu}$.

First, assume that both optimal actions are serving queues $\bmu$ and $\bnu$.
If $q_{\bmu}>0$ and $q_{\bnu}>0$, we obtain
\begin{align*}
V_{k+1}(\bq)+V_{k+1}(\bq+2\bfe_{\bmu}+\bfe_{\bnu}) =& c_{\bmu}+c_{\bnu} + \beta\E[V_{k}( \bq+\bA-\bfe_{\bmu}-\bfe_{\bnu} )]
+ c_{\bmu}+c_{\bnu} + \beta\E[V_{k}( \bq+\bA+\bfe_{\bmu} )]\\
\leq& c_{\bmu}+c_{\bnu} + \beta\E[V_{k}(\bq+\bA) ]
+ c_{\bmu}+c_{\bnu} + \beta\E[V_{k}(\bq+\bA-\bfe_{\bnu})]\\
\leq& V_{k+1}(\bq+\bfe_{\bmu}+\bfe_{\bnu})+V_{k+1}(\bq+\bfe_{\bmu});
\end{align*}
here the first inequality follows from the induction hypothesis (substituting $\bq+\bA-\bfe_{\bmu}-\bfe_{\bnu}$ for $\bq$ in \eqref{eq:for_three_different} for $V_k$) and the second inequality follows from the definition of the value iteration. 
On the other hand, if $q_{\bmu}>0$ and $q_{\bnu}=0$, we have
\begin{align*}
V_{k+1}(\bq)+V_{k+1}(\bq+2\bfe_{\bmu}+\bfe_{\bnu}) =& c_{\bmu} + \beta\E[V_k(\bq+\bA-\bfe_{\bmu})]
+ c_{\bmu}+c_{\bnu} + \beta\E[V_k(\bq+\bA+\bfe_{\bmu})]\\
\leq& c_{\bmu}+\beta\E[ V_k(\bq+\bA) ]
+ c_{\bmu}+c_{\bnu} + \beta\E[V_{k}(\bq+\bA )] \\
\leq& V_{k+1}(\bq+\bfe_{\bmu})+V_{k+1}(\bq+\bfe_{\bmu}+\bfe_{\bnu}),
\end{align*}
where the first inequality follows from the induction hypothesis (substituting $\bq+\bA-\bfe_{\bmu}$ for $\bq$ in \eqref{eq:for_two} for $V_k$) and the second inequality follows from the definition of the value iteration. 
Lastly, if $q_{\bmu}=0$, we obtain
\begin{align*}
V_{k+1}(\bq)+V_{k+1}(\bq+2\bfe_{\bmu}+\bfe_{\bnu}) =& c_{\bnu}\cdot \indic_{\{q_{\bnu}>0\}} + \E[V_{k}( (\bq-\bfe_{\bnu})^++\bA )]
+ c_{\bmu}+c_{\bnu} + \E[V_{k}( \bq+\bA+\bfe_{\bmu} )]\\
\leq& c_{\bmu}+c_{\bnu}\cdot \indic_{\{q_{\bnu}>0\}} + \E[V_{k}( (\bq-\bfe_{\bnu})^++\bA )]
+ c_{\bmu}+c_{\bnu} + \E[V_{k}( \bq+\bA )]\\
\leq& V_{k+1}(\bq+\bfe_{\bmu})+V_{k+1}(\bq+\bfe_{\bmu}+\bfe_{\bnu}),
\end{align*}
where the first inequality follows from Proposition~\ref{prop:trivial_boundary} and the second inequality follows from the definition of the value iteration.

Second, suppose that both optimal actions are serving queues $\brho$ and $\bomega$.
Then, we have
\begin{align*}
V_{k+1}(\bq)+V_{k+1}(\bq+2\bfe_{\bmu}+\bfe_{\bnu}) &\\
& \hspace*{-1.8in} =2(c_{\brho}\cdot \indic_{\{q_{\brho}>0\}}+c_{\bomega}\cdot \indic_{\{q_{\bomega}>0\}}) + \beta\E[ V_k( (\bq-\bfe_{\brho}-\bfe_{\bomega})^++\bA )]
+ \beta\E[ V_k( (\bq-\bfe_{\brho}-\bfe_{\bomega})^++\bA+2\bfe_{\bmu}+\bfe_{\bnu} )]\\
& \hspace*{-1.8in} \leq 2(c_{\brho}\cdot \indic_{\{q_{\brho}>0\}}+c_{\bomega}\cdot \indic_{\{q_{\bomega}>0\}})
+ \beta\E[ V_k( (\bq-\bfe_{\brho}-\bfe_{\bomega})^++\bA+\bfe_{\bmu}+\bfe_{\bnu} )]
+ \beta\E[ V_k( (\bq-\bfe_{\brho}-\bfe_{\bomega})^++\bA+\bfe_{\bmu} )]\\
& \hspace*{-1.8in} \leq 
V_{k+1}(\bq+\bfe_{\bmu}+\bfe_{\bnu})+V_{k+1}(\bq+\bfe_{\bmu});
\end{align*}
here the first inequality follows from the induction hypothesis (substituting $(\bq-\bfe_{\brho}-\bfe_{\bomega})^++\bA$ for $\bq$ in \eqref{eq:for_three_different} for $V_k$) and the second inequality follows from the definition of the value iteration.

Finally, assume that the optimal action on $\bq$ is to serve queues $\brho$ and $\bomega$,
and the optimal action on $\bq+2\bfe_{\bmu}+\bfe_{\bnu}$ is to serve queues $\bmu$ and $\bnu$.
Then, we obtain
\begin{align*}
V_{k+1}(\bq+2\bfe_{\bmu}+\bfe_{\bnu})+V_{k+1}(\bq) =& c_{\bmu}+c_{\bnu} + \beta\E[ V_k(\bq+\bA+\bfe_{\bmu})]
+ c_{\brho}\cdot \indic_{\{q_{\brho}>0\}}+c_{\bomega}\cdot \indic_{\{q_{\bomega}>0\}}\\
&\quad +\beta\E[ V_k((\bq-\bfe_{\brho}-\bfe_{\bomega})^++\bA) ]\\
=& ( c_{\bmu}+c_{\bnu}+ c_{\brho}\cdot \indic_{\{q_{\brho}>0\}}+c_{\bomega}\cdot \indic_{\{q_{\bomega}>0\}})
+ \E[V_{k}(\bQ+\bz)+V_{k}(\bQ+\bw)], \\
V_{k+1}(\bq+\bfe_{\bmu}+\bfe_{\bnu})+V_{k+1}(\bq+\bfe_{\bmu}) \geq& c_{\bmu}+c_{\bnu} + \beta\E[ V_k(\bq+\bA)]
+ c_{\brho}\cdot \indic_{\{q_{\brho}>0\}}+c_{\bomega}\cdot \indic_{\{q_{\bomega}>0\}}\\
&\quad +\beta\E[ V_k((\bq-\bfe_{\brho}-\bfe_{\bomega})^++\bA+\bfe_{\bmu}) ]\\
=& ( c_{\bmu}+c_{\bnu}+ c_{\brho}\cdot \indic_{\{q_{\brho}>0\}}+c_{\bomega}\cdot \indic_{\{q_{\bomega}>0\}})
+ \E[V_{k}(\bQ+\bx)+V_{k}(\bQ+\by)], 
\end{align*}
where 
$\bQ\eqdef (\bq-\bfe_{\brho}-\bfe_{\bomega})^++\bA$,
$\bx\eqdef \bfe_{\bmu}$,
$\by\eqdef \bq-(\bq-\bfe_{\brho}-\bfe_{\bomega})^+$,
$\bz\eqdef \bfe_{\bmu}+\bq-(\bq-\bfe_{\brho}-\bfe_{\bomega})^+$,
$\bw\eqdef \bzero$.
We also have $\bx,\by,\bz,\bw\in\{0,1\}^{4}$, $\bx\leq \bfe_{\bmu}+\bfe_{\bnu}$, $\by\leq\bfe_{\brho}+\bfe_{\bomega}$, $\bx+\by=\bz+\bw$,
and thus we obtain
\begin{align*}
V_{k+1}(\bq+\bfe_{\bmu}+\bfe_{\bnu})+V_{k+1}(\bq+\bfe_{\bmu}) &\geq  V_{k+1}(\bq+2\bfe_{\bmu}+\bfe_{\bnu})+V_{k+1}(\bq) ,
\end{align*}
which follows from Lemma~\ref{LEMMA:AUX_INEQUALITY_ALT}.

Hence, \eqref{eq:for_three_different} holds for all value functions.

\mss{\section{Proof of Lemma~\ref{LEMMA:AUX_INEQUALITY_ALT}}
\label{sec:pf:aux_ineqaulity_alt}}
We prove Lemma~\ref{LEMMA:AUX_INEQUALITY_ALT} by induction.
The proposition is true for $V_{0}$ because $V_{0}(\bq)=0$ for all $\bq\in\Int_+^{4}$.
Now, suppose that the proposition holds for the $k$-th value function $V_k$ and that
$\bx,\by\in\{0,1\}^4$ and $\bz,\bw\in \Ints^{4}$ 
satisfy the assumptions in Lemma~\ref{LEMMA:AUX_INEQUALITY_ALT}.
We then show that \eqref{eq:aux_inequality} holds for the $(k+1)$-th value function $V_{k+1}$ in each of the following cases which depend on the optimal actions of $k$-th value iteration on $\bq+\bz$ and $\bq+\bw$.

\begin{enumerate}
\item[(I):] \textit{Both optimal actions are the same}.
Without loss of generality, assume that both optimal actions are serving queues $\bmu$ and $\bnu$.
Then, the right-hand side of \eqref{eq:aux_inequality} becomes
\begin{align}
V_{k+1}(\bq+\bz)+V_{k+1}(\bq+\bw) =& c_{\bmu}\cdot \indic_{\{q_{\bmu}+z_{\bmu}>0\}}+ c_{\bnu}\cdot \indic_{\{q_{\bnu}+z_{\bnu}>0\}}
+ \beta\E[ V_k( (\bq+\bz-\bfe_{\bmu}-\bfe_{\bnu})^++\bA  ) ] \nonumber \\
&\quad +\ c_{\bmu}\cdot \indic_{\{q_{\bmu}+w_{\bmu}>0\}}+ c_{\bnu}\cdot \indic_{\{q_{\bnu}+w_{\bnu}>0\}}
+ \beta\E[ V_k( (\bq+\bw-\bfe_{\bmu}-\bfe_{\bnu})^++\bA  ) ] \nonumber \\
=& c_{\bmu}(\indic_{\{q_{\bmu}+z_{\bmu}>0\}}+\indic_{\{q_{\bmu}+w_{\bmu}>0\}})
+ c_{\bnu}(\indic_{\{q_{\bnu}+z_{\bnu}>0\}}+\indic_{\{q_{\bnu}+w_{\bnu}>0\}}) \nonumber \\
&\quad + \beta\E[V_k(\bQ+\bz^\prime)+V_k(\bQ+\bw^\prime)], \label{eq:RHS_V_n_same_action}
\end{align}
where 
$\bQ\eqdef(\bq-\bfe_{\bmu}-\bfe_{\bnu})^++\bA$,
$\bz^\prime\eqdef(\bq+\bz-\bfe_{\bmu}-\bfe_{\bnu})^+-(\bq-\bfe_{\bmu}-\bfe_{\bnu})^+$,
$\bw^\prime\eqdef(\bq+\bw-\bfe_{\bmu}-\bfe_{\bnu})^+-(\bq-\bfe_{\bmu}-\bfe_{\bnu})^+$.
On the other hand, by the definition of the value iteration, we obtain
\begin{align}
V_{k+1}(\bq+\bx)+V_{k+1}(\bq+\by) \geq & c_{\bmu}\cdot \indic_{\{q_{\bmu}+x_{\bmu}>0\}}+ c_{\bnu}\cdot \indic_{\{q_{\bnu}+x_{\bnu}>0\}}
+ \beta\E[ V_k( (\bq+\bx-\bfe_{\bmu}-\bfe_{\bnu})^++\bA  ) ] \nonumber \\
&\quad +  c_{\bmu}\cdot \indic_{\{q_{\bmu}+y_{\bmu}>0\}}+ c_{\bnu}\cdot \indic_{\{q_{\bnu}+y_{\bnu}>0\}}
+ \beta\E[ V_k( (\bq+\by-\bfe_{\bmu}-\bfe_{\bnu})^++\bA ) ] \nonumber \\
=& c_{\bmu}(\indic_{\{q_{\bmu}+x_{\bmu}>0\}}+\indic_{\{q_{\bmu}+y_{\bmu}>0\}})
+ c_{\bnu}(\indic_{\{q_{\bnu}+x_{\bnu}>0\}}+\indic_{\{q_{\bnu}+y_{\bnu}>0\}}) \nonumber \\
&\quad + \beta\E[V_k(\bQ+\bx^\prime)+V_k(\bQ+\by^\prime)], \label{eq:LHS_V_n_same_action}
\end{align} 
where 
$\bx^\prime\eqdef(\bq+\bx-\bfe_{\bmu}-\bfe_{\bnu})^+-(\bq-\bfe_{\bmu}-\bfe_{\bnu})^+$,
$\by^\prime\eqdef(\bq+\by-\bfe_{\bmu}-\bfe_{\bnu})^+-(\bq-\bfe_{\bmu}-\bfe_{\bnu})^+$.
Next, comparing the right-hand side of the above equations, since $\bx+\by=\bz+\bw$ and $\bx,\by\in\{0,1\}^4$ and $\bz,\bw\in \Ints^{4}$, we have
\begin{align*}
c_{\bmu}(\indic_{\{q_{\bmu}+x_{\bmu}>0\}}+ \indic_{\{q_{\bmu}+y_{\bmu}>0\}})
	&=c_{\bmu}(\indic_{\{q_{\bmu}+z_{\bmu}>0\}}+ \indic_{\{q_{\bmu}+w_{\bmu}>0\}}),\\
c_{\bnu}(\indic_{\{q_{\bnu}+x_{\bnu}>0\}}+ \indic_{\{q_{\bnu}+y_{\bnu}>0\}})
	&=c_{\bnu}(\indic_{\{q_{\bnu}+z_{\bnu}>0\}}+ \indic_{\{q_{\bnu}+w_{\bnu}>0\}}),
\end{align*}
and thus the first two terms on the right-hand side of  \eqref{eq:RHS_V_n_same_action} and \eqref{eq:LHS_V_n_same_action} are the same. \\
For the last part, it is readily verified that $\bx,\by\in\{0,1\}^4$ and $\bz,\bw\in \Ints^{4}$, $\bx^\prime\leq\bfe_{\bmu}+\bfe_{\bnu}$, $\by^\prime\leq\bfe_{\brho}+\bfe_{\bomega}$, and  $\bx^\prime+\by^\prime=\bz^\prime+\bw^\prime$,
which implies that 
\begin{align*}
	V_k(\bQ+\bz^\prime)+V_k(\bQ+\bw^\prime)&\leq V_k(\bQ+\bx^\prime)+V_k(\bQ+\by^\prime)
\end{align*}
because $V_k$ satisfies Lemma~\ref{LEMMA:AUX_INEQUALITY_ALT} (induction hypothesis).
Hence, \eqref{eq:aux_inequality} holds for $V_{k+1}$ in this case.

\item[(II):] \textit{Two optimal actions are different from each other}.
Without loss of generality, assume that the optimal action on $(\bq+\bz)$ is serving queues $\bmu$ and $\bnu$, and the optimal action on $(\bq+\bw)$ is serving queues $\brho$ and $\bomega$.
Then, as before, the right-hand side and left-hand side of \eqref{eq:aux_inequality} become
\begin{align}
V_{k+1}(\bq+\bz)+V_{k+1}(\bq+\bw) =& c_{\bmu}\cdot \indic_{\{q_{\bmu}+z_{\bmu}>0\}}+ c_{\bnu}\cdot \indic_{\{q_{\bnu}+z_{\bnu}>0\}}
+ c_{\brho}\cdot \indic_{\{q_{\brho}+w_{\brho}>0\}}+ c_{\bomega}\cdot \indic_{\{q_{\bomega}+w_{\bomega}>0\}} \nonumber \\
&\quad  +\beta\E[ V_k( \bQ+\bz^\prime  ) + V_k( \bQ+\bw^\prime  ) ] , \label{eq:RHS_V_n_different_action} \\
V_{k+1}(\bq+\bx)+V_{k+1}(\bq+\by) \geq& c_{\bmu}\cdot \indic_{\{q_{\bmu}+x_{\bmu}>0\}}+ c_{\bnu}\cdot \indic_{\{q_{\bnu}+x_{\bnu}>0\}}
+ c_{\brho}\cdot \indic_{\{q_{\brho}+y_{\brho}>0\}}+ c_{\bomega}\cdot \indic_{\{q_{\bomega}+y_{\bomega}>0\}} \nonumber \\
&\quad +\beta\E[ V_k( \bQ+\bx^\prime  ) + V_k( \bQ+\by^\prime ) ], \label{eq:LHS_V_n_differnt_action}
\end{align}
where we can have two separated cases.
Namely, Case (IIa): $\bx+\by-\bfe_{\bmu}-\bfe_{\bnu}-\bfe_{\brho}-\bfe_{\bomega}\ge 0$, and then
\begin{align*}
\bQ &= \bq+\bA, \quad
\bx^\prime =\by^\prime  = 0,\\
\bz^\prime &= (\bq+\bz-\bfe_{\bmu}-\bfe_{\bnu})^+-\bq,
 \quad \bw^\prime = (\bq+\bw-\bfe_{\brho}-\bfe_{\bomega})^+-\bq;
\end{align*}
and Case (IIb): at least one of the components of  $\bx+\by-\bfe_{\bmu}-\bfe_{\bnu}-\bfe_{\brho}-\bfe_{\bomega}$ is negative, and then
\begin{align*}
\bQ &= (\bq+\bx+\by-\bfe_{\bmu}-\bfe_{\bnu}-\bfe_{\brho}-\bfe_{\bomega})^++\bA,\\
\by^\prime &=(\bq+\bx-\bfe_{\bmu}-\bfe_{\bnu})^+-(\bq+\bx+\by-\bfe_{\bmu}-\bfe_{\bnu}-\bfe_{\brho}-\bfe_{\bomega})^+, \\
\bx^\prime & = (\bq+\by-\bfe_{\brho}-\bfe_{\bomega})^+-(\bq+\bx+\by-\bfe_{\bmu}-\bfe_{\bnu}-\bfe_{\brho}-\bfe_{\bomega})^+,\\
\bz^\prime &= (\bq+\bz-\bfe_{\bmu}-\bfe_{\bnu})^+-(\bq+\bx+\by-\bfe_{\bmu}-\bfe_{\bnu}-\bfe_{\brho}-\bfe_{\bomega})^+,
\\ \bw^\prime &= (\bq+\bw-\bfe_{\brho}-\bfe_{\bomega})^+-(\bq+\bx+\by-\bfe_{\bmu}-\bfe_{\bnu}-\bfe_{\brho}-\bfe_{\bomega})^+.
\end{align*}


For Case (IIa): the conditions for $\bx'$ and $\by'$ are trivially satisfied; if $\bq+\bz-\bfe_{\bmu}-\bfe_{\bnu}\ge 0$ and $\bq+\bw-\bfe_{\brho}-\bfe_{\bomega}\ge 0$, then we have
$\bx'+\by'=\bz'+\bw'$, and the induction assumption implies
\begin{align*}
	V_k(\bQ+\bz^\prime)+V_k(\bQ+\bw^\prime)&\leq V_k(\bQ+\bx^\prime)+V_k(\bQ+\by^\prime).
\end{align*}
These conditions also imply that $\indic_{\{q_{\bmu}+z_{\bmu}>0\}}=\indic_{\{q_{\bnu}+z_{\bnu}>0\}}=\indic_{\{q_{\brho}+w_{\brho}>0\}}= \indic_{\{q_{\bomega}+w_{\bomega}>0\}}=1$. In summary, 
\begin{align*}
V_{k+1}(\bq+\bx)+V_{k+1}(\bq+\by) \geq V_{k+1}(\bq+\bz)+V_{k+1}(\bq+\bw).
\end{align*}

For Case (IIb), where at least one of the components of $\bq+\bz-\bfe_{\bmu}-\bfe_{\bnu}$ and $\bq+\bw-\bfe_{\brho}-\bfe_{\bomega}$ is negative, we need to treat it differently. Without loss of generality, let us assume that it is the $\bmu$ component that is negative, i.e. $q_{\bmu}+z_{\bmu}=0$. Thus,  $\bx'+\by'=\bz'+\bw'-\bfe_{\bmu}$, and $\bw'$ will have positive component in $\bmu$. Define, 
$\bz'' = \bz'$ and $\bw''=\bw'-\bfe_{\bmu}$. Obviously, $\bx'+\by'=\bz''+\bw''$, and the induction assumption implies
\begin{align*}
	V_k(\bQ+\bz'')+V_k(\bQ+\bw'')&\leq V_k(\bQ+\bx^\prime)+V_k(\bQ+\by^\prime).
\end{align*}
Meanwhile, it is easy to see that
\begin{align*}
	V_k(\bQ+\bz^\prime)+V_k(\bQ+\bw^\prime)&\leq c_{\bmu}+ V_k(\bQ+\bz'')+V_k(\bQ+\bw'').
\end{align*}
Combine the above two together, we have
\begin{align*}
	V_k(\bQ+\bz^\prime)+V_k(\bQ+\bw^\prime)&\leq c_{\bmu}+ V_k(\bQ+\bx^\prime)+V_k(\bQ+\by^\prime).
\end{align*}
This implies that, since $\indic_{\{q_{\bmu}+z_{\bmu}>0\}}=0$,
\begin{align}
\label{eqn:one_extra_in_z}
	c_{\bmu}\cdot \indic_{\{q_{\bmu}+z_{\bmu}>0\}}+\beta\ex[V_k(\bQ+\bz^\prime)+V_k(\bQ+\bw^\prime)]&\le c_{\bmu}\cdot \indic_{\{q_{\bmu}+x_{\bmu}>0\}}+ \beta\ex[V_k(\bQ+\bx^\prime)+V_k(\bQ+\by^\prime)].
\end{align}
Similar steps can be applied to the other components, yielding
\begin{align*}
V_{k+1}(\bq+\bx)+V_{k+1}(\bq+\by) \geq V_{k+1}(\bq+\bz)+V_{k+1}(\bq+\bw).
\end{align*}

For Case (IIb): it is easy to check that the conditions for $\bx'$ and $\by'$ are satisfied; similar to the case that $\bq+\bz-\bfe_{\bmu}-\bfe_{\bnu}\ge 0$ and $\bq+\bw-\bfe_{\brho}-\bfe_{\bomega}\ge 0$, we then have
$\bx'+\by'=\bz'+\bw'$, and thus the induction assumption leads us to the desired inequality; in the case that at least one of the components of the two vectors is negative, we apply a similar treatment as above where \eqref{eqn:one_extra_in_z} is obtained. 


\end{enumerate}

Hence, the $(k+1)$-th value function $V_{k+1}$ satisfies \eqref{eq:aux_inequality} and Lemma~\ref{LEMMA:AUX_INEQUALITY_ALT} holds.
\end{APPENDICES}



\bibliographystyle{informs2014} 
\bibliography{references,temp} 



\end{document}